\documentclass{amsart}
\pdfoutput=1
\PassOptionsToPackage{pdftex}{graphics}
\usepackage{subfig}
\usepackage{tabularx}
\usepackage{enumerate}
\usepackage[latin1]{inputenc}

\usepackage{amssymb,mathrsfs,amsmath}
\usepackage{tabularx}
\usepackage{fancyhdr,fancybox}
\usepackage{epstopdf}
\usepackage{epsfig}
\usepackage{psfrag}
\usepackage{hyperref}
\usepackage{enumerate}

\usepackage{graphicx,color}
\usepackage{mathpazo}
\usepackage{hyperref}
\usepackage{tabularx}
\usepackage{subfig}
\def\conv{\text{\normalfont conv}}
\def\green{\color{black}}
\def\black{\color{black}}
\def\yellow{\color{black}}

\def\RR{\mathbb R}
\def\moda{\text{\normalfont mod }}
\def\CC{\mathbb C}
\def\NN{\mathbb N}

\def\im{\text{\normalfont im~}}
\def\re{\text{\normalfont re~}}

\def\leq{\leqslant}

\def\geq{\geqslant}

\newtheorem{theorem}{Theorem}[section]

\newtheorem{lemma}[theorem]{Lemma}
\newtheorem{notation}[theorem]{Notation}
\newtheorem{proposition}[theorem]{Proposition}
\newtheorem{corollary}[theorem]{Corollary}
\newtheorem{remark}[theorem]{Remark}
\newtheorem{example}[theorem]{Example}

\title{Geometrical aspects of expansions in complex bases}
\author{Anna Chiara Lai}
\address{Dipartimento di Scienze di Base e Applicate per L'ingegneria, Sapienza Universit\`a di Roma}
\begin{document}
\begin{abstract}
 We study the set of the representable numbers in base $q=pe^{i\frac{2\pi}{n}}$ with $\rho>1$ and $n\in \NN$ and with digits in a arbitrary finite real alphabet $A$.
We give a geometrical description of the convex hull of the representable numbers in base $q$ and alphabet $A$ and an explicit characterization of its extremal points.
A characterizing condition for the convexity of the set of representable numbers is also shown.
\end{abstract}
\maketitle
\section{Introduction}\label{s complex intro}

In this paper we deal with expansions with digits in arbitrary alphabets and bases of the form $p e^{\frac{2\pi i}{n}}$ with $p>1$ and $n\in \NN$, namely we are interested in developments in power series of the form
 \begin{equation}\label{intro e1}
 \sum_{j=1}^\infty \frac{x_j}{q^j}
\end{equation}
with the coefficients $x_j$ belonging to a finite set of positive real values $A$ named \emph{alphabet} and the ratio $q=p e^{\frac{2\pi i}{n}}$, named \emph{base}.
The assumption $p>1$ ensures the convergence of (\ref{intro e1}), being the base $q$ greater than $1$ in modulus. When a number $x$ satisfies
$$x=\sum_{j=1}^\infty \frac{x_j}{q^j},$$
 for a sequence $(x_j)_{j\geq1}$ with digits in the alphabet $A$, we say that $x$ is \emph{representable} in base $q$ and alphabet $A$ and we call $(x_j)_{j\geq1}$ a \emph{representation} or \emph{expansion} of $x$.

The first number systems in complex base seem to be those in base $2i$ with alphabet $\{0,1,2,3\}$ and the one in base $-1+i$ and alphabet $\{0,1\}$,
 respectively introduced by Knuth in \cite{Knu60} and by Penney in \cite{Pen65}. After that many papers were devoted to representability with bases belonging to
larger and larger classes of complex numbers, e.g. see \cite{KS75} for the Gaussian integers in the form $-n\pm i$ with $n\in \NN$, \cite{KK81} for the quadratic
fields and \cite{DK88} for the general case. Loreti and Komornik pursued the work in \cite{DK88} by introducing a greedy algorithm for the expansions in complex base
with non rational argument \cite{KL07}. In the eighties a parallel line of research was developed by Gilbert. In \cite{Gil81} he described the fractal nature of the set of
the representable numbers, e.g. the set of representable in base $-1+i$ with digits $\{0,1\}$ coincides with the fascinating space-filling twin dragon curves \cite{Knu71}.
Hausdorff dimension of some set of representable numbers was calculated in \cite{Gil84} and a weaker notion of self-similarity was introduced for the study of the boundary of the
representable sets \cite{Gil87}. Complex base numeration systems and in particular the geometry of the set of representable numbers have been widely studied by the point of view of
their relations with iterated function systems and tilings of the complex plane, too. For a survey on the topology of the tiles associated to bases belonging to quadratic fields we
refer to \cite{AT04}.

We study the convex hull of the set of representable numbers by giving first a geometrical description then an explicit characterization of its extremal points. We also show a characterizing
condition for the convexity of the set of representable numbers.

Expansions in complex base have several applications. For example, in the context of computer arithmetics, the interesting property of these numerations systems is that
they allow multiplication and division of complex base in a unified manner, without treating real and imaginary part separately --- see \cite{Knu71}, \cite{Gil84} and \cite{FS03}.
 Representation in complex base have been also used in cryptography with the purpose of speeding up  onerous computations such as modular exponentiations~\cite{DJM85} and multiplications
over elliptic curves~\cite{Sol00}. Finally we refer to \cite{Pic02} for a dissertation on the applications of the numerations in complex base to the compression of images on fractal tilings.

\subsection*{Organization of the paper} Most of the arguments of this paper laying on geometrical properties, in Section \ref{s geoback} we show some results on complex plane geometry.
In Section \ref{s convex char}  we characterize the shape and the extremal points of the convex hull of the set of representable numbers. In Section \ref{s complex repr} we give a necessary
and sufficient condition to have a \emph{convex} set of representable numbers, this property being sufficient for a full representability of complex numbers.

\section{Geometrical background}\label{s geoback}
By using the isometry between $\CC$ and $\RR^2$ we extend to $\CC$ some definitions which are proper of the plane geometry.
\smallskip
Elements of $\CC$ are considered vectors (or sometimes points) and we endow $\CC$ with the scalar product
$$\mathbf u\cdot \mathbf v:=|\mathbf u||\mathbf v|\cos(\arg\mathbf u-\arg \mathbf v).$$
\green\begin{remark}\label{product}\black
 $$\mathbf u\cdot \mathbf v\leq 0$$
if and only if
$$\cos(\arg \mathbf u-\arg \mathbf v) \leq0.$$
\end{remark}\black

 A polygon in the complex plane is the bounded region of $\CC$ contained in a closed chain of segments, the \emph{edges}, whose endpoints are the \emph{vertices}.
 If two adjacent edges belong to the same line, namely if they are adjacent and parallel, then they are called \emph{consecutive} and their common endpoint is a \emph{degenerate vertex}.
If a vertex is not degenerate, it is called \emph{extremal point}.
The index operations on the vertices $\mathbf v_1,\dots,\mathbf v_H\in \CC$ of a polygon are always considered modulus their number, namely
$$\mathbf v_{h+s}:=\mathbf v_{h+s~\moda H}$$
so that for instance  $\mathbf v_{0}=\mathbf v_{H}$ and $\mathbf v_{H+1}=\mathbf v_{1}$.

Consider an edge whose endpoints are two vertices $\mathbf v_{h-1}$ and $\mathbf v_h$ and its \emph{normal vector}
\begin{equation}\label{n1}
 \mathbf n_h=(\mathbf v_h-\mathbf v_{h-1})^\bot:=-i(\mathbf v_h-\mathbf v_{h-1})
\end{equation}

A set of vertices $\{\mathbf v_1,\dots,\mathbf v_H\}$ is \emph{counter-clockwise ordered} if there exists $s\in\{0,\dots,H-1\}$ such that for every $h=1,\dots,H-1$
\begin{equation}
 \arg \mathbf n_{h+s}\leq \arg \mathbf n_{h+s+1}.
\end{equation}
A set of vertices $\{\mathbf v_1,\dots,\mathbf v_H\}$ is \emph{clockwise ordered} if there exists $s\in\{0,\dots,H-1\}$ such that for every $i=1,\dots,H-1$
\begin{equation}\label{argn}
 \arg \mathbf n_{h+s}\geq \arg \mathbf n_{h+s+1}.
\end{equation}
\begin{remark}\label{rmkconvex}
A polygon is convex if and only if its vertices are clockwise or counter-clockwise ordered (see \cite{Gems} for the $\RR^2$ case, the complex case readily follows by employing the
isometry $(x,y)\mapsto x+iy$).
\end{remark}
  We are now interested in establishing condition on a point $x\in \CC$ to belong to a convex polygon.
 The following result is an adapted version of the exterior criterion for the Point-In-Polygon problem (see for instance \cite{PIP}).

\green\begin{proposition}\label{proconv}\black
Let $P$ be a convex polygon whose vertices are $\mathbf v_1,\dots,\mathbf v_H$.
Then a complex value $x$ belongs to $P$ if and only if for every $h=1,\dots, H$
$$ (x-\mathbf v_h)\cdot \mathbf n_h\leq 0.$$
\end{proposition}\black

%
%

\smallskip
 The \emph{convex hull} of a set $X\subset \CC$ is the smallest convex set containing $X$  and it is denoted by using the symbol $\text{\normalfont conv}(X)$.
 When $X$ is finite, its convex hull is a convex polygon whose vertices are in $X$.

%
We now study the convex hull of the set $P\cup (P+\mathbf t)$, being $P$ a convex polygon, $\mathbf t \in \CC$ and $P+\mathbf t:=\{x+\mathbf t \mid x\in P\}$.
\green\begin{remark}\black
 If $\mathbf t\in \CC$ then
$$\arg \mathbf t^\bot=\left(\arg \mathbf t + \frac{3\pi}{2}\right)~\moda 2\pi$$
and
$$\arg(-\mathbf t)^\bot=\left(\arg \mathbf t + \frac{\pi}{2}\right)~\moda 2\pi$$
\end{remark}\black


%

\begin{theorem}\label{t trans}
 Let $P$ be a convex polygon with counter-clockwise ordered vertices in $\{\mathbf v_1,\dots,\mathbf v_{H}\}$ and let $\mathbf t\in \CC$.
Denote $h_1$ and $h_2$ the indices respectively satisfying
\begin{align}
\arg \mathbf n_{h_1}\leq \arg -\mathbf t^\bot < \arg \mathbf n_{h_1+1} \label{nh1}\\
 \arg \mathbf n_{h_2}\leq \arg \mathbf t^\bot < \arg \mathbf n_{h_2+1} \label{nh2}
\end{align}
Then the convex hull of $P\cup(P+\mathbf t)$ is a polygon whose vertices are:
\begin{equation}\label{t trans e0}
 \mathbf v_{h_1},\dots,\mathbf v_{h_2},\mathbf v_{h_2}+\mathbf t,\dots,\mathbf v_{h_1-1}+\mathbf t,\mathbf v_{h_1}+\mathbf t.
\end{equation}
\end{theorem}
\begin{remark}
 As the index operations on the vertices are considered modulus $H$, if $h_1>h_2$ the expression in (\ref{t trans e0}) means
\begin{equation}\label{t trans e0exy}
 \mathbf v_{h_1},\dots,\mathbf v_{H},\mathbf v_1,\dots,\mathbf v_{h_2},\mathbf v_{h_2}+\mathbf t,\dots,\mathbf v_{h_1-1}+\mathbf t,\mathbf v_{h_1}+\mathbf t.
\end{equation}
conversely if $h_1<h_2$ the extended version of (\ref{t trans e0}) is
\begin{equation}\label{t trans e1exy}
 \mathbf v_{h_1},\dots,\mathbf v_{h_2},\mathbf v_{h_2}+\mathbf t,\dots,\mathbf v_{H}+\mathbf t,\mathbf v_1+\mathbf t,\dots,\mathbf v_{h_1-1}+\mathbf t,\mathbf v_{h_1}+\mathbf t.
\end{equation}
\end{remark}

\begin{proof}
As $P$ is a convex polygon, its vertices are either clockwise or counter-clockwise ordered. We may assume without loss of generality the latter case and, in particular, that
\begin{equation}\label{t trans e1}
 \arg(\mathbf n_h)\leq \arg(\mathbf n_{h+1})
\end{equation}
for every $h=1,\dots,H-1$. Therefore $h_1$ and $h_2$ are well defined. By definition $h_1\not=h_2$, hence we may distinguish the cases $h_1>h_2$ and $h_1<h_2$. We discuss only the latter case, because the proves are similar.
 We use the symbol $\overline P$ to denote the polygon whose vertices are listed in (\ref{t trans e0}), in particular we set
\begin{equation*}
 \overline {\mathbf v}_h=\begin{cases}
                          \mathbf v_{h+h_1-1} &\text{if } h=1,\dots,h_2-h_1+1\\
                          \mathbf v_{h+h_1-2} +\mathbf t &\text{if } h=h_2-h_1+2,\dots,H+2\\
                         \end{cases}
\end{equation*}
Remark that by shifting the vertices of $\overline P$ of $s:=H-h_1+2$ positions, we obtain the ordered list
$$\mathbf v_1+\mathbf t,\dots,\mathbf v_{h_1}+\mathbf t,\mathbf v_{h_1},\dots,\mathbf v_{h_2},\mathbf v_{h_2}+\mathbf t,\dots,\mathbf v_{H}+\mathbf t$$
and, in particular,
$$\overline{\mathbf v}_{h+s}=\begin{cases}
                                  \mathbf v_h+\mathbf t &\text{ if } h=1,\dots,h_1\\
                                  \mathbf v_{h-1} &\text{ if } h=h_1+1,\dots,h_2+1\\
                  \mathbf v_{h-2}+\mathbf t &\text{ if } h=h_2+2,\dots,H+2.
                                 \end{cases}
$$
Hence
$$\overline{\mathbf n}_{h+s}=\begin{cases}
                                  \mathbf n_h &\text{ if } h=1,\dots,h_1\\
                                  (-t)^\bot   &\text{ if } h=h_1+1\\
                                  \mathbf n_{h-1} &\text{ if } h=h_1+2,\dots,h_2+1\\
                                  (t)^\bot   &\text{ if } h=h_2+2\\
                  \mathbf n_{h-2} &\text{ if } h=h_2+3,\dots,H+2.
                                 \end{cases}
$$
Thus the definition of $h_1$ and of $h_2$, together with (\ref{t trans e1}), implies that for every $h=1,\dots, H+1$
\begin{equation}
\arg \overline{\mathbf n}_{h+s}\leq \arg\overline{\mathbf n}_{h+s+1}
\end{equation}
namely the vertices of $\overline P$ are counter-clockwise ordered. In view of Remark \ref{rmkconvex} we may deduce that $\overline P$ is convex and, in particular,
$$\overline P=\conv \{\overline{\mathbf v}_h\mid h=1,\dots,H+2\}.$$
Now we want to show
$$\overline P=\conv(P\cup(P+\mathbf t))$$
by double inclusion. Since $P$ is convex, we have
$$P=\conv\{\mathbf v_h \mid h=1,\dots,H\},$$
$$P+t=\conv\{\mathbf v_h+\mathbf t\mid h=1,\dots,H\}$$
and
$$\conv(P\cup (P+\mathbf t))=\conv\{\mathbf v_h,\mathbf v_h+\mathbf t\mid h=1,\dots,H\},$$
therefore
$$\overline P= \conv\{ \overline{\mathbf v}_h\mid h=1,\dots,H\} \subseteq\conv(P\cup(P+\mathbf t)).$$
 To prove the other inclusion, it suffices to show that $P\cup (P+\mathbf t) \subseteq \overline P$. In view of Proposition \ref{proconv},
 this is equivalent to prove that for every $y\in P\cup(P+\mathbf t)$ and for every $h=1,\dots, 2H$
\begin{equation}\label{olo}
 (y-\overline{ \mathbf v}_h)\cdot \overline{\mathbf n}_{h}\leq 0.
\end{equation}
Now consider $y\in P\cup(P +\mathbf t)$ and remark that
$$y=x+\alpha \mathbf t$$
for some $x\in P$ and $\alpha\in\{0,1\}$. Therefore we may rewrite (\ref{olo}) as follows
\begin{equation}\label{lol}
 (x+\alpha \mathbf t-\overline{ \mathbf v}_h)\cdot \overline{\mathbf n}_{h}\leq 0
\end{equation}
for every $h=1,\dots,H+2$.
 First of all remark that, by Proposition \ref{proconv}, for every $h=1,\dots,H$
\begin{equation}
 (x-\mathbf v_h)\cdot \mathbf n_h\leq 0
\end{equation}
and, consequently,
\begin{equation}
\begin{split}
(x-\mathbf v_h)\cdot \mathbf n_{h+1}&=(x-\mathbf v_{h+1})\cdot \mathbf n_{h+1}+(\mathbf v_{h+1}-\mathbf v_h)\cdot \mathbf n_{h+1}\\
&=(x-\mathbf v_{h+1})\cdot \mathbf n_{h+1}+(\mathbf v_{h+1}-\mathbf v_h)\cdot (\mathbf v_{h+1}-\mathbf v_h)^\bot\\
&=(x-\mathbf v_{h+1})\cdot \mathbf n_{h+1}\\
&\leq 0;
\end{split}
\end{equation}
Therefore,
\begin{equation}\label{est}
 \max\{\cos(\arg (x-\mathbf v_{h})-\arg \mathbf n_{h}),\cos(\arg (x-\mathbf v_{h})-\arg \mathbf n_{h+1})\}\leq 0.
\end{equation}
Now,
\begin{itemize}
 \item if $h=1$ then the definition of $h_1$, namely
$$\arg \mathbf n_{h_1}\leq \arg(-\mathbf t)^\bot < \arg \mathbf n_{h_1+1},$$
and (\ref{est}) imply
\begin{align*}
 \cos(\arg (x-\mathbf v_{h})&-\arg (-\mathbf t)^\bot) \\
&\leq\max\{\cos(\arg (x-\mathbf v_{h_1})-\arg \mathbf n_{h_1}),\cos(\arg (x-\mathbf v_{h_1})-\arg \mathbf n_{h_1+1})\}\\
&\leq0.
\end{align*}
Then
$$(x+\alpha \mathbf t -\overline{\mathbf v}_{1})\cdot \overline{\mathbf n}_1=(x-\mathbf v_{h_1})\cdot (-\mathbf t)^\bot+\alpha \mathbf t\cdot (-\mathbf t)^\bot \leq 0.$$
%
\item
If $h=2,\dots,h_2-h_1$ then
$$(x+\alpha \mathbf t-\overline {\mathbf v}_{h})\cdot \overline{\mathbf n}_h=
 (x-\mathbf v_{h+h_1-1})\cdot \mathbf n_{h+h_1-1}+\alpha \mathbf t\cdot \mathbf n_{h+h_1-1}\leq 0,$$
because $\mathbf t\cdot \mathbf n_{h}\leq 0$ whenever $h_1\leq h\leq h_2$. Indeed
$$\left(\arg \mathbf t +\frac{\pi}{2}\right)~\moda 2\pi=\arg (-\mathbf t^\bot)\leq \arg \mathbf n_h\leq \arg \mathbf t^\bot=\left(\arg \mathbf t +\frac{3\pi}{2}\right)~\moda 2\pi$$
hence
$$\cos(\arg \mathbf t-\arg \mathbf n_h)\leq \max\left\{\cos\left(\frac{\pi}{2}\right),\cos\left(\frac{3\pi}{2}\right)\right\}=0.$$
\item
If $h=h_2-h_1+1$ then
\begin{equation*}
(x+\alpha \mathbf t-\overline {\mathbf v}_{h})\cdot \overline{\mathbf n}_h= (x+\alpha\mathbf t-\mathbf v_{h_2}-\mathbf t)\cdot \mathbf t^\bot=(x-\mathbf v_{h_2})\cdot \mathbf t^\bot\leq 0
\end{equation*}
Indeed
$$\arg \mathbf n_{h_2}\leq \arg \mathbf t^\bot < \arg \mathbf n_{h_2+1},$$
and (\ref{est}) imply
\begin{align*}
 \cos(\arg (x-\mathbf v_{h})&-\arg (\mathbf t)^\bot)\leq \\
&\leq\max\{\cos(\arg (x-\mathbf v_{h_2})-\arg \mathbf n_{h_2}),\cos(\arg (x-\mathbf v_{h_2})-\arg \mathbf n_{h_2+1})\}\\
&\leq0.
\end{align*}
\item Finally if $h=h_2-h_1+2,\dots,H+2$ then
\begin{align*}
(x+\alpha \mathbf t -\overline {\mathbf v}_{h})\cdot \overline{\mathbf n}_h&=
 (x+\alpha \mathbf t-\mathbf v_{h+h_1-2}-\mathbf t)\cdot \mathbf n_{h_1-2+h}\\
& \leq (x-\mathbf v_{h+h_1-2})\cdot \mathbf n_{h_1-2+h}-(1-\alpha)\mathbf t\cdot  \mathbf n_{h_1-2+h}\\
&\leq 0
\end{align*}
because $1-\alpha\geq 0$ and $\mathbf t\cdot  \mathbf n_{h_1-2+h}\geq 0$ for every $h=h_2-h_1+2,\dots,H+2$.
\end{itemize}
Hence (\ref{lol}) holds for every $x\in P$ and $\alpha\in\{0,1\}$ and the proof is complete.
\end{proof}
\green\begin{corollary}\label{t trans c1}\black
 Let $P$ be a convex polygon with $l$ edges and let $\mathbf t\in \CC$. Then $\conv(P \cup (P+\mathbf t))$ has $H+2$ (possibly consecutive) edges, in particular $l$ edges of $\conv(P\cup(P+\mathbf t))$ are parallel
to the edges of $P$ and $2$ edges are parallel to $\mathbf t$.
\end{corollary}\black

\begin{proof}
 It immediately follows by the list of vertices given in (\ref{t trans e0}).
\end{proof}

\green\begin{corollary}\label{t trans c2}\black
 Let $P$ be a convex polygon with $e$ extremal points and let $\mathbf t\in \CC$. Then:
\begin{enumerate}[(a)]
\item if $\mathbf t$ is not parallel to any edge of $P$ then $\conv(P\cup(P+\mathbf t))$ has $e+2$ extremal points;
\item if $\mathbf t$ is parallel to $1$ edge of $P$ then $\conv(P\cup(P+\mathbf t))$ has $e+1$ extremal points;
\item if $\mathbf t$ is parallel to $2$ edges of $P$ then $\conv(P\cup(P+\mathbf t))$ has $e$ extremal points.
\end{enumerate}
\end{corollary}\black

\begin{proof}
Set $l$ the number of edges of $P$ and $d:=l-e$ the number of
degenerate vertices. By definition, a vertex $\mathbf v$ is
degenerate if the normal vectors of its adjacent edges have the same
argument.
 It follows by (\ref{t trans e0}) that the normal vectors of $\conv(P\cup (P+\mathbf t))$ are the following:
\begin{equation*}
 \mathbf n_{h_1},\dots,\mathbf n_{h_2-1},\mathbf t^\bot,\mathbf n_{h_2},\dots,\mathbf n_{h_1-1},-\mathbf t^\bot
\end{equation*}
with $h_1$ and $h_2$ satisfying:
\begin{align*}
\arg \mathbf n_{h_1}\leq \arg -\mathbf t^\bot <\arg \mathbf n_{h_1+1} \\
 \arg \mathbf n_{h_2}\leq \arg \mathbf t^\bot <\arg \mathbf n_{h_2+1}
\end{align*}
Thus by denoting $l_t$ the number of vertices of $P_t$, by $d_t$ the number of the degenerate vertices and by $e_t$ the number of extremal points,
 by Corollary \ref{t trans c1} we have that $l_t=l+2$ and :
\begin{equation*}
 d_t=\begin{cases}
      d &\text{ if $\mathbf t^\bot$ is neither parallel to $\mathbf n_{h_1}$ nor to $\mathbf n_{h_2-1}$};\\
      d+1&\text{ if $\mathbf t^\bot$ is parallel to $\mathbf n_{h_1}$ or to $\mathbf n_{h_2-1}$};\\
      d+2&\text{ if $\mathbf t^\bot$ is parallel to $\mathbf n_{h_1}$ and to $\mathbf n_{h_2-1}$}.\\
     \end{cases}
\end{equation*}
Hence thesis follows by the relation $e_t=l_t-d_t$.
\end{proof}

We conclude this section with the following result on the convexity of $P\cup P+\mathbf t$.

\begin{proposition}\label{l:geometrical hq}
 Let $P$ be a convex polygon and assume that $2$ edges of $P$ are parallel to the real axis on the complex plane and that their length is equal to $1$. Consider
$t_1,\dots,t_m\in\RR$ such that
\begin{equation}\label{orf}
 t_1<\cdots<t_m.
\end{equation}
Then
$$\max_{i=1,\dots,m-1}{t_{i+1}-t_i}\leq 1$$
 if and only if
$$\bigcup_{i=1}^m (P+t_i)$$
is a convex set.
\end{proposition}
\begin{proof}

\vskip0.5cm\noindent
\emph{Only if part.}

\smallskip\noindent
Suppose $t_{i_1+1}-t_{i_1}>1$ for some $i_1\in\{1,\dots,m-1\}$. Let $\mathbf v_1$ and $\mathbf v_2$ two extremal points of $P$  such that
$$\re(\mathbf v_1)\leq \re (\mathbf v_2)$$
and
$$\im(\mathbf v_1)=\im(\mathbf v_2)$$
By the inequality above, the edge with endpoints $\mathbf v_1$ and $\mathbf v_2$ is parallel to the real axis and, consequently,
\begin{equation}\label{reff}
 \re(\mathbf v_2)-\re(\mathbf v_1)\leq 1<t_{i_1+1}-t_{i_1}.
\end{equation}
Now define for every $i=1,\dots,m$
$$\mathbf e_i:=\{\mathbf v_1+x\mid x\in[t_i,t_i+\re \mathbf v_2-\re \mathbf v_1]\}$$
and remark that the convexity of $P$ implies
$$\bigcup_{i=1}^mP+t_i\cap \{\mathbf v_1+t\mid t\in \RR\}=\bigcup_{i=1}^m\mathbf e_i$$
By (\ref{orf}), for every $\alpha\in(0,1)$ and every $i=1,\dots,m-1$
$$\alpha (\mathbf v_2+t_i)+(1-\alpha)(\mathbf v_1+t_{i+1})\in \bigcup_{i=1}^m\mathbf e_i$$
if and only if
$$\alpha (\mathbf v_2+t_i)+(1-\alpha)(\mathbf v_1+t_{i+1})\in \mathbf e_i\cap\mathbf e_{i+1}.$$
Now, setting
$\mathbf x_1:=\mathbf v_2+t_{i_1}$ and $x_1:=\mathbf v_1+t_{i_1+1}$, we have that for every $\alpha\in(0,1)$
$$\alpha \mathbf x_1+(1-\alpha)\mathbf x_2\in \conv\left(\bigcup_{i=1}^m (P+t_i)\right)$$
and
$$\alpha \mathbf x_1+(1-\alpha)\mathbf x_2\in \{\mathbf v_1+t\mid t\in \RR\}$$
but, in view of (\ref{reff}),
$$\mathbf e_{i_1}\cap \mathbf e_{i_1+1}=\emptyset$$
hence
$$\alpha \mathbf x_1+(1-\alpha)\mathbf x_2\not\in \bigcup_{i=1}^m(P+t_i).$$
Therefore $\bigcup_{i=1}^m(P+t_i)$ is not a convex set.

\vskip0.5cm\noindent
\emph{If part}

\smallskip\noindent
Let $m_P$ and $M_P$ be such  that $P$ is a subset of $\{x\in \CC | m_p\leq \im(x) \leq M_P\}$. For every $m_P \leq y \leq M_P$ we consider the set
$$I_y:=\{x\in \RR \mid x+i y \in P\}$$
As $P$ is convex, $I_y$ is an interval, whose endpoints are denoted by $a_y$ and $b_y$.
The convexity of  $P$ also implies
$\mid I_y\mid\geq 1$;
 therefore $t_{i+1}-t_i\leq 1$ implies
$$[a_y+t_1,b_y+ t_m]=\bigcup_{i=1}^m [a_y+t_i,b_y+t_i].$$
and
 \begin{equation} \label{l:convtran e1}
 x+iy\in \bigcup_{i=1}^m (P+t_i) \quad \text{ if and only if }\quad x\in [a_y+t_{1},b_y+t_m].
\end{equation}
We want to prove that $\bigcup_{i=1}^m (P+t_i)$ is a convex set by showing that it contains any convex combination of its points. So fix
$$\mathbf x_1,\mathbf x_2 \in\bigcup_{i=1}^m (P+t_i).$$
If $\mathbf x_1$ and $\mathbf x_2$ are both in $P+t_i$ for some $i=1,\dots,m$ the convexity of $P$ implies the thesis.
Otherwise suppose $\mathbf x_1\in P+t_{i_1}$ and $\mathbf x_2\in P+t_{i_2}$ with $t_{i_1}<t_{i_2}$
and consider a convex combination $\alpha \mathbf x_1 + (1-\alpha)\mathbf x_2$, with $\alpha\in[0,1]$.
Remark that $\mathbf x_2 \in P+t_{i_2}$ implies  $\mathbf x_2-(t_{i_2}-t_{i_1}) \in P+t_{i_1}$ and, consequently,
$$\mathbf x:=\alpha \mathbf x_1+ (1-\alpha)(\alpha\mathbf x_2-\mathbf t)\in P+t_{i_1}.$$
Therefore
\begin{equation*}
 \alpha \mathbf x_1 + (1-\alpha)\mathbf x_2 = \mathbf x +(1-\alpha)\mathbf t= x+i y+(1-\alpha)(t_{i_2}-t_{i_1}).
\end{equation*}
 for some $x\in I_y+t_{i_1}$ and $y\in[m_P,M_P]$. Hence
\begin{align*}
 a_y+t_{1}\leq a_y+t_{i_1}\leq x+(1-\alpha)(t_{i_2}-t_{i_1})&\leq b_y+t_{i_1}+(1-\alpha)(t_{i_2}-t_{i_1})\\
&\leq b_y+t_m.
\end{align*}
In view of (\ref{l:convtran e1}) we finally get $\alpha \mathbf x_1 + (1-\alpha)\mathbf x_2 \in \bigcup_{i=1}^m P+t_i$ and hence the thesis.
\end{proof}

\section{Characterization of the convex hull of representable numbers}\label{s convex char}
In this section we investigate the shape of the convex hull of the set of representable numbers in base $p e^{\frac{2\pi}{n}i}$ and with alphabet $A$. We adopt the following notations.

 \green\begin{notation}\black
We use the symbol $\Lambda_{n,p,A}$ to denote the set of representable numbers in base $q_{n,p}:=p e^{\frac{2\pi}{n}i}$ and with alphabet $A$, namely
$$\Lambda_{n,p,A}=\left\{\sum_{j=1}^\infty \frac{x_j}{ q^{j}_{n,p}} \mid x_j\in A\right\}.$$
We set
$$X_{n,p}:=\left\{\sum_{k=0}^{n-1} {x}_k q_{n,p}^k \mid {x}_k\in \{0,1\}\right\}$$
and
$$P_{n,p}:=\text{\normalfont conv}(X_{n,p}).$$
   \end{notation}

\begin{remark}\black
   As $X_{n,p}$ is finite, $P_{n,p}$ is a polygon.
\end{remark}\black
The following result represents a first simplification of our problem: in fact the characterization of the convex hull of the infinite set $\Lambda_{n,p,A}$ is showed to be equivalent to the study of $P_{n,p}$.

\green\begin{lemma}[Farkas' Lemma]\black
 Let $A\in \RR^{m\times n}$ and $ b\in \RR^{n}$. The system $A\mu=b$ admits a non-negative solution if and only if for every $u\in \RR^m$ the inequality $A^\top \cdot u\geq 0$
 implies $b^\top \cdot u\geq 0$.
\end{lemma}\black
\green\begin{lemma}\label{lfar}\black
 Let $q\in \CC$, $n\in \NN$ and $\lambda_k\in[\lambda_{min},\lambda_{max}]$ for every $k=0,\dots,n-1$, $\lambda_{min},\lambda_{max}\in\RR$ and
$$S:=\{(x_0^{(h)},\dots,x_{n-1}^{(h)})\mid x_k^{(h)}\in\{\lambda_{min},\lambda_{max}\};~ k=1,\dots,n;~ h=1\dots,2^{n}\}$$
be the set of sequences of length $k$ and with digits in $\{\lambda_{min},\lambda_{max}\}$.

Then there exist $\mu_1,\dots,\mu_{2^n}\geq 0$ such that
$$\sum_{h=1}^{2^n}\mu_h=1$$ and
\begin{equation}\label{far}
 \sum_{k=0}^{n-1}q^k\lambda_k=\sum_{h=1}^{2^n}\mu_h\sum_{k=0}^{n-1}q^kx_k^{(h)}
\end{equation}
\end{lemma}\black

\begin{proof}
 Consider the linear system with $2^n$ indeterminates $\mu_1,\dots,\mu_{2^n}$ and with $n+1$ equations
$$\begin{cases}
\displaystyle{\sum_{h=1}^{2^n}\mu_h x_{k}^{(h)}=\lambda_k} \quad \text { for }k=0,\dots, n-1;\\
\displaystyle{\sum_{h=1}^{2^n}\mu_h=1.}
  \end{cases}$$
We may rewrite the above system in the form
$$A \mu =b$$
where the $h$-th column of $A$ satisfies $A_h^\top=(x_1^{(h)},\dots,x_k^{(h)},1)$ and $b^\top=(\lambda_1,\dots,\lambda_k,1)$. In order to apply  Farkas' Lemma consider $u\in \RR^{n+1}$ such that
$$A_h^\top \cdot u=\sum_{k=1}^nx_k^{(h)}u_k+u_{n+1}\geq 0$$
for every $h=1,\dots,2^n$. We recall that $\lambda_k\in[\lambda_{min},\lambda_{max}]$ for every $k$, then we may consider the sequence $(x_1^{(h)},\dots,x^{(h)}_n)\in S$ such that for every
 $k=1,\dots,n$
$$x^{(h)}_k\geq \lambda_k \quad \text{ if and only if }\quad u_k\geq0.$$
 Consequently
$$\sum_{k=1}^n(\lambda_k-x_k^{(h)})u_k\geq 0$$
and
$$b^\top \cdot u=\sum_{k=1}^n\lambda_ku_k+u_{n+1}\geq \sum_{k=1}^nx_k^{(h)}u_k+u_{n+1}\geq 0.$$
By Farkas' Lemma, there exist $\mu_1,\dots,\mu_{2^n}\geq0$ with $\sum_{h=1}^{2^n}\mu_h=1$ such that
$$\sum_{h=1}^{2^n}\mu_h x_{k}^{(h)}=\lambda_k;$$
for every $k=1,\dots,n$ and, consequently, $\mu_1,\dots,\mu_{2^n}$ also satisfy
$$\sum_{k=0}^{n-1}q^k\lambda_k=\sum_{k=0}^{n-1}q^k\sum_{h=1}^{2^n}\mu_h x_k^{(h)}=\sum_{h=1}^{2^n}\mu_h\sum_{k=0}^{n-1}q^kx_k^{(h)}.$$
\end{proof}

\green\begin{proposition}\label{t convex p1}\black
 For every $n\geq 1$, $p>1$ and $q_{n,p}=p e^{\frac{2\pi}{n}i}$:
\begin{equation}
 \text{\normalfont conv}(\Lambda_{n,p,A})=\frac{\max A-\min A}{p^n-1}\cdot P_{n,p}+\frac{1}{p^n-1}\sum_{k=0}^{n-1}\min A~q^{k}_{n,p}.
\end{equation}
\end{proposition}\black

\begin{proof}
Fix $n$ and $p$ and, in order to lighten the notations, set $q=q_{n,p}$, $X_n=X_{n,p}$ and $P_{n}=P_{n,p}$.
Consider
$$x=\sum_{j=1}^\infty \frac{x_j}{q^j}\in \Lambda_{n,p,A}.$$
 As $q^n=p^n$,
\begin{equation}\label{t convex p1 e1}
\begin{split}
 \sum_{j=1}^\infty \frac{x_j}{q^{j}}
&=\sum_{k=0}^{n-1}q^k\sum_{j=1}^\infty\frac{ x_{jn-k}}{ p^{jn}}\\
&=\sum_{k=0}^{n-1}q^k \lambda_k
\end{split}
\end{equation}
with
$$\lambda_k:=\sum_{j=1}^\infty \frac{x_{jn-k}}{ p^{jn}}\in  \left[\frac{\min A}{p^n-1},\frac{\max A}{p^n-1}\right].$$
 Therefore by Lemma \ref{lfar} any element of $\Lambda_{n,p,A}$
is a convex combination of complex numbers that can be written in
 the form $$\sum_{k=0}^{n-1}q^k \tilde x_k$$
 with $\tilde x_k \in \left\{\dfrac{\min A}{p^n-1},\dfrac{\max A}{p^n-1}\right\}$. Thus
\begin{equation*}
\begin{split}
 \text{\normalfont conv}\left(\Lambda_{n,p,A}\right)&=\text{\normalfont conv}\left( \left\lbrace \sum_{k=0}^{n-1}q^k \tilde x_k\mid \tilde x_k\in \left\lbrace\frac{\min A}{p^n-1},\frac{\max A}{p^n-1}\right\rbrace\right\rbrace \right)\\
&=\frac{\max A-\min A}{p^n-1}\text{\normalfont conv}( X_{n,p})+\frac{1}{p^n-1}\sum_{k=0}^{n-1}\min A~ q^{k}.
\end{split}
 \end{equation*}
\end{proof}

We now give a geometrical description of
 $\conv(\Lambda_{n,p,A})$.

\green\begin{theorem}[Convex hull of the representable numbers in complex base]\label{t convex}\black
 For every $n\geq 1$, $p>1$ and alphabet $A$, the set  $\text{\normalfont conv}(\Lambda_{n,p,A})$ is a polygon with the following properties:
\begin{enumerate}[(a)]
 \item the edges are pairwise parallel to $q^0,\dots,q^{n-1}$, where $q=p e^{\frac{2\pi i}{n}}$;
\item if $n$ is odd then $\text{\normalfont conv}(\Lambda_{n,p,A})$ has $2n$ extremal points;
\item if $n$ is even then $\text{\normalfont conv}(\Lambda_{n,p,A})$ has $n$ extremal points.
\end{enumerate}
\end{theorem}
 \black
\begin{proof}
Fix $n$ and $p$ and, in order to lighten the notations, set $q=q_{n,p}$, $X_n=X_{n,p}$ and $P_{n}=P_{n,p}$.
 Our proof is based on showing $P_{n}$ to have properties (a), (b) and (c);
indeed these properties are invariant by rescaling and translation and, by Proposition \ref{t convex p1}, they extend from $P_{n}$ to $\conv(\Lambda_{n,p,A})$.

\smallskip
 We divide the proof in three parts.

\vskip0.5cm
\noindent\emph{Part 1. The edges of $P_{n}$ are pairwise parallel to $q^0,\dots,q^{n-1}$.}

\smallskip
\noindent
To the end of studying $P_n$, we consider the sets
\begin{equation*}
 X_{m}:=\left\lbrace\sum_{j=0}^{m-1}x_j q^j \mid x_j \in \{0,1\}\right\rbrace;
\end{equation*}
for $m=1,\dots,n$. We have
\begin{equation}\label{t convex e1}
 \begin{cases}
  X_{1}=\{0,1\};\\
  X_{m}=X_{m-1}\cup (X_{m-1}+q^{m-1})
 \end{cases}
\end{equation}
and consequently
\begin{equation}\label{t convex e2}
 \begin{cases}
  P_{1}=[0,1];\\
  P_{m}=\text{\normalfont conv}\left(P_{m-1}\cup (P_{m-1}+q^{m-1})\right).
 \end{cases}
\end{equation}
We remark that $P_{1}$
can be looked at as a polygon with two vertices and with two overlapped edges that are parallel to $q^0$. By iteratively applying Corollary \ref{t trans c1}
we deduce that $P_{n}$ has pairwise parallel edges and every couple of edges is either parallel  to $q^0$ or to any of the successive translation, i.e. $q^1,\dots,q^{n-1}$.

\vskip0.5cm
\noindent\emph{Part 2. If $n$ is odd then $P_{n}$ has $2n$ extremal points.}

\smallskip
\noindent
First remark that $n$ odd implies that $q^j$ and $q^k$ are not parallel for every $j\not=k$.
 We showed above that for every $m=1,\dots,n$ the edges of the polygon $P_{m-1}$ are parallel to $q^0,\dots,q^{m-2}$ and, consequently,
the translation $q^{m-1}$ is not parallel to any edge.
 Hence, denoting $e_{m}$ the number of the extremal points of $P_{m}$, the first part of Corollary
 \ref{t trans c1} implies that $e_{m}$ is defined by the recursive relation:
\begin{equation*}
\begin{cases}
 &e_{1}=2\\
 &e_{m}=e_{m-1}+2
\end{cases}
\end{equation*}
for every $m=1,\dots,n$. Hence $e_{n}=2n$.

\vskip0.5cm
\noindent\emph{Part 3. If $n$ is even then $P_{n}$ has $n$ extremal points.}

\smallskip
\noindent If $n$ is even then $q_n^{m+n/2}$ is parallel to $q_n^m$ for every $m=1,\dots,n/2$.
Since $P_{m}$ has pairwise parallel edges, we deduce by (a) and (c) in Corollary \ref{t trans c2} that $e_{m}$ is defined by the relation:
\begin{equation*}
\begin{cases}
 e_{0}=0;&\\
 e_{m}=e_{m-1}+2 \quad &\text{if } m=1,\dots,n/2;\\
 e_{m}=e_{m-1} \quad &\text{if } m=n/2+1,\dots,n;
\end{cases}
\end{equation*}
hence $e_{n}=n$ and this concludes the proof.
\end{proof}

\begin{example}
 If $n=3$ and if $p>1$ then  $\conv(\Lambda_{n,p,A})$ is an hexagon.

 If $n=4$ and  if $p>1$ then $\conv(\Lambda_{n,p,A})$ is a rectangle.
\end{example}

\begin{figure}[h]
\centering
 \subfloat[ $\conv(X_{3,2^{1/3}})$ ]{
\includegraphics[scale=0.5]{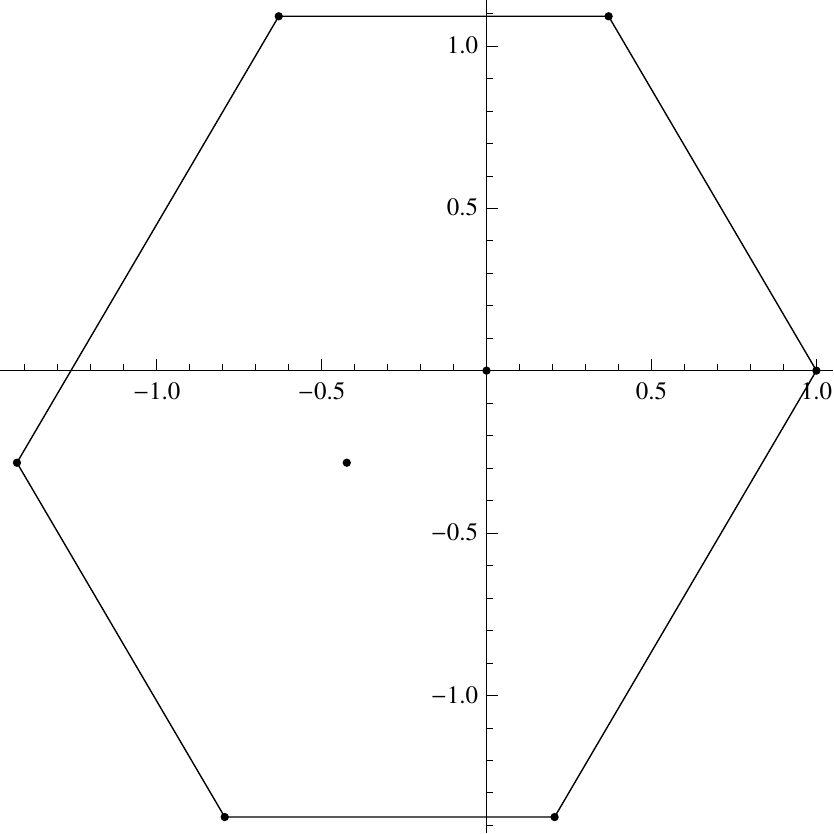}
 }\hskip2.5cm
 \subfloat[$\conv(X_{4,2^{1/4}})$ ]{\includegraphics[scale=0.5]{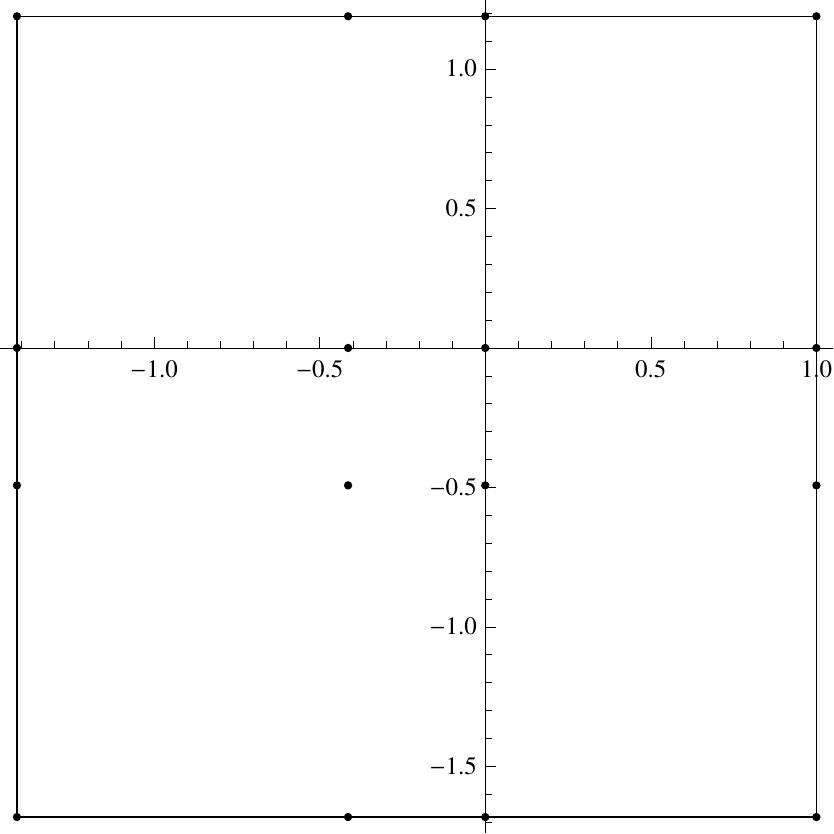}}
\caption{Convex hulls of $X_{3,2^{1/3}}$ and of $X_{4,2^{1/4}}$. Remark that when $A=\{0,1\}$ then $P_{n,p}=\text{\normalfont conv} (X_{n,p})$ coincides with $\dfrac{1}{p^n-1}\text{\normalfont conv}(\Lambda_{n,p,A})$.}
\end{figure}

\begin{figure}
\centering
 \subfloat[$n=3$ ]{
\includegraphics[scale=0.5]{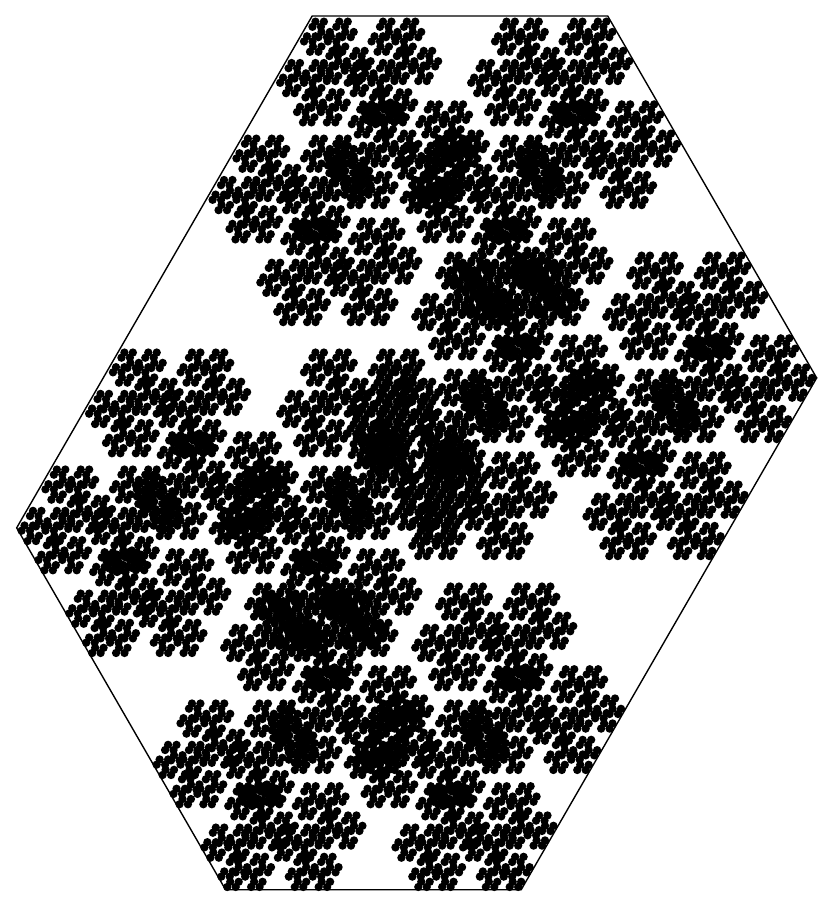}
 }
\subfloat[$n=4$ ]{
\includegraphics[scale=0.4]{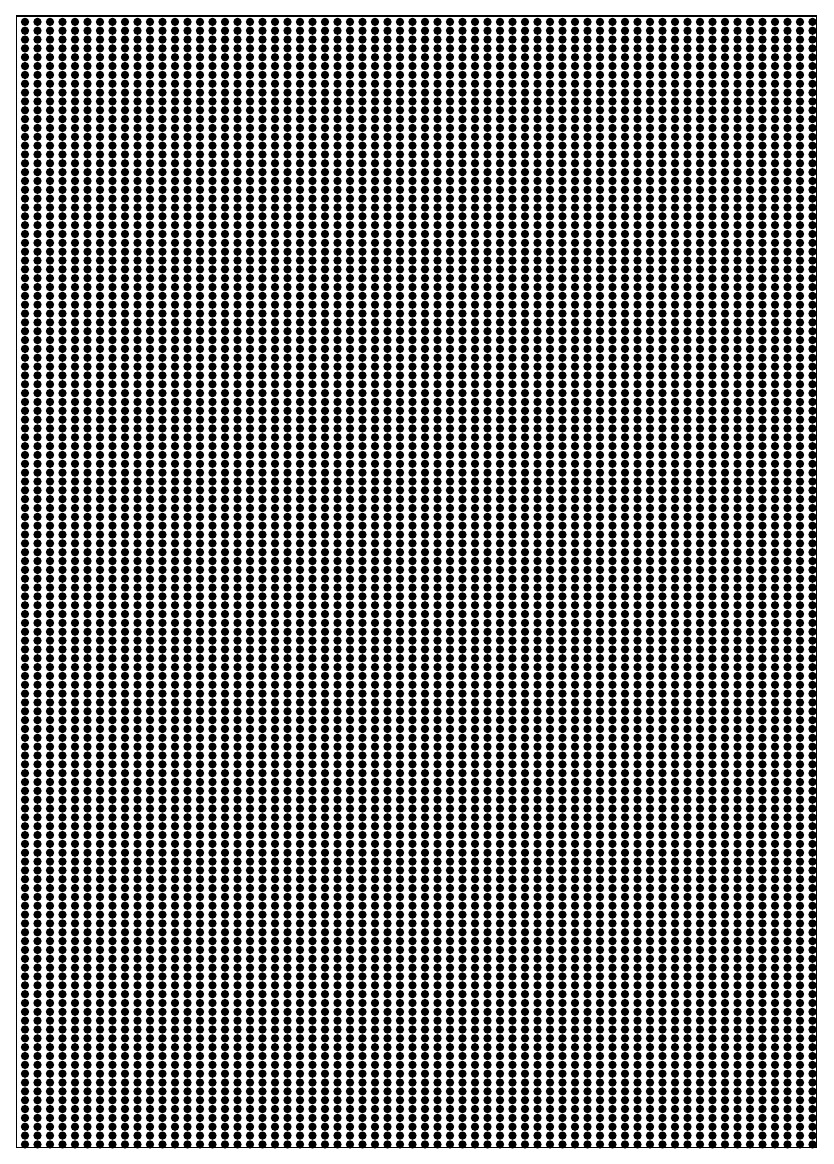}
 }
\subfloat[$n=5$]{
\includegraphics[scale=0.5]{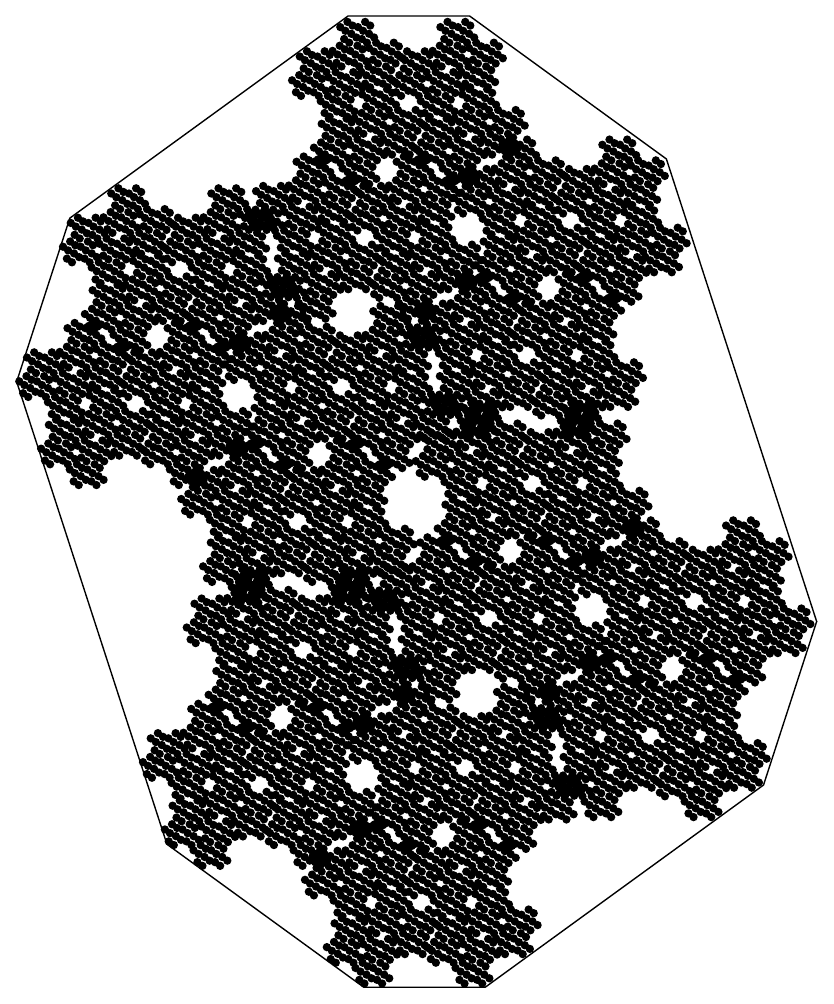}
 }\\
 \subfloat[$n=6$]{
\includegraphics[scale=0.5]{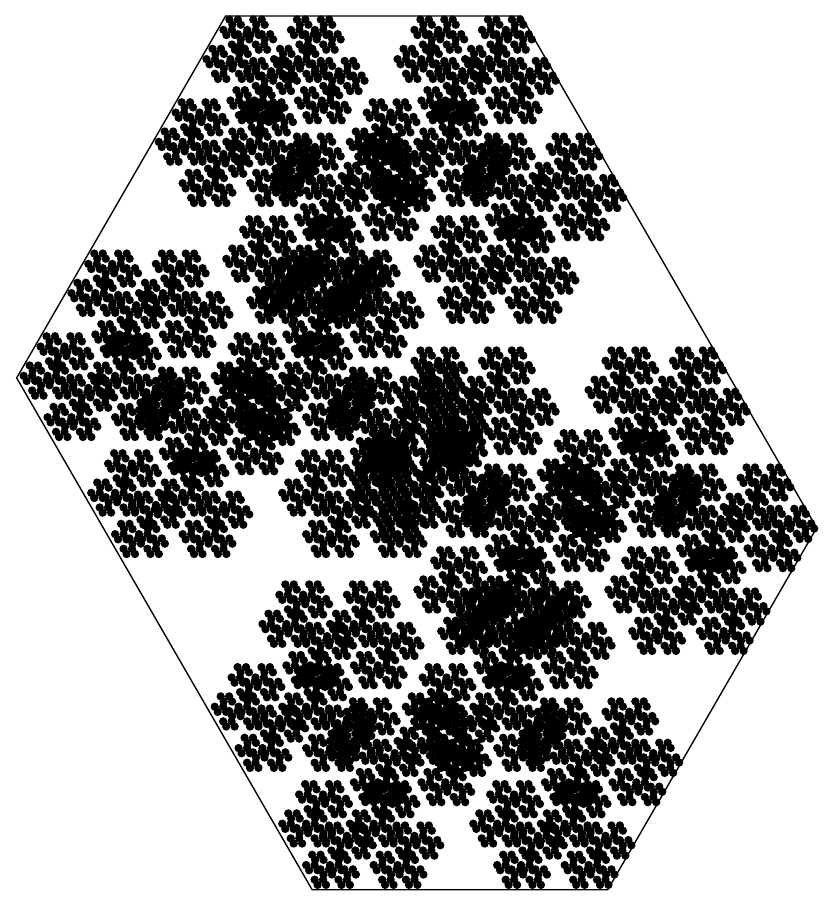}
 }
\subfloat[$n=7$]{
\includegraphics[scale=0.5]{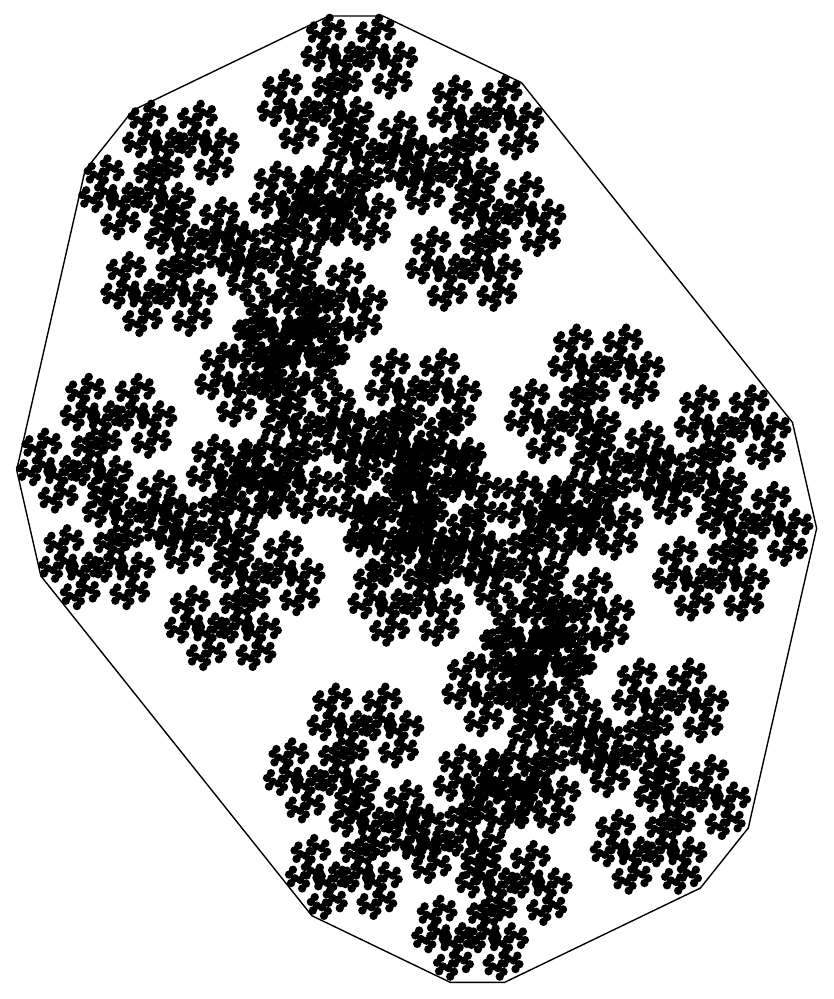}
 }
\subfloat[$n=8$ ]{
\includegraphics[scale=0.5]{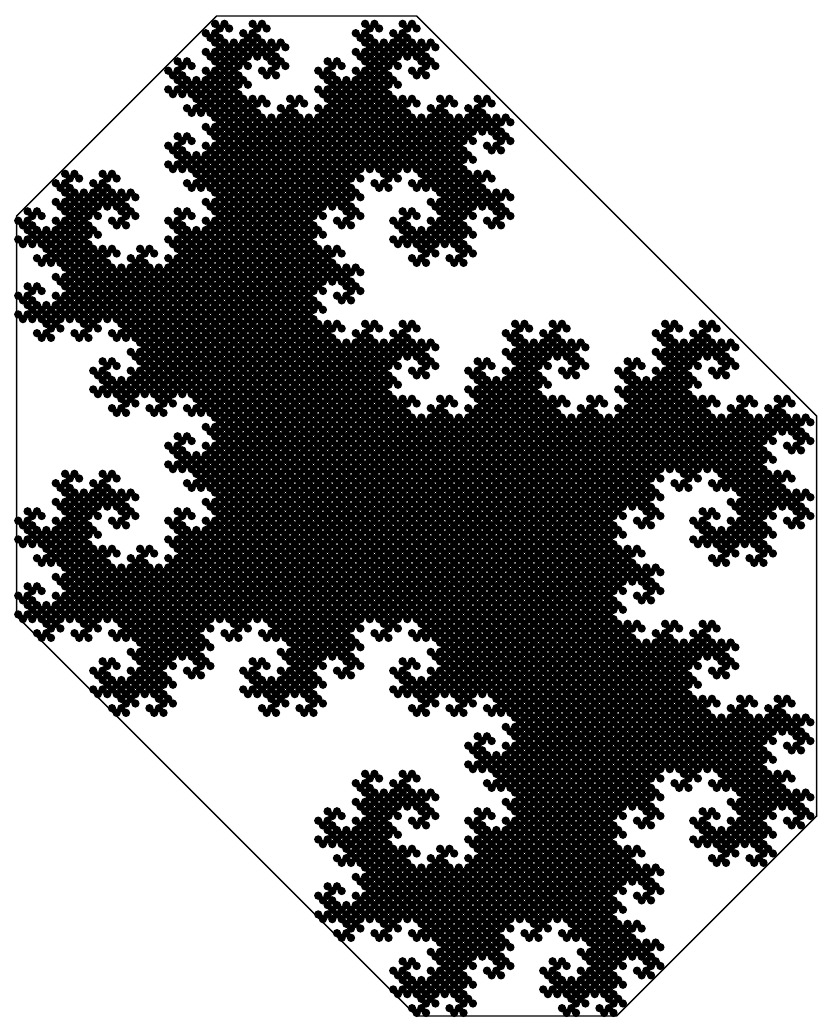}
 }
\caption{The sets $\Lambda_{n,2^{1/2},\{0,1\}}$ with $n=3,\dots,8$ and their convex hulls. $\Lambda_{n,2^{1/2},\{0,1\}}$
is approximated with the set of expansions with length $14$.
Remark that $q_{8,2^{1/2}}=1+i$ is a Gaussian integer that has been studied, for instance, in \cite{Gil81}.}
\end{figure}

After establishing the shape of $\conv(\Lambda_{n,p,A})$,
 we are now interested on the explicit characterization of its extremal points.
 By Proposition \ref{t convex p1}, this is equivalent to characterize the extremal points of $P_{n,p}$ and we shall focus on this problem. Let us see some examples.

\green\begin{example}\label{t charcon ex 1}\black
 Set $p>1$ and $q:=q_{3,p}=pe^{\frac{2\pi}{3}i}$. By a direct computation, the set of extremal points f $P_{3,p}$, say $\mathcal E(P_{3,p})$, is
\begin{equation*}
 \begin{split}
 \mathcal E(P_{3,p})=\{1\,,1+q\,,q\,,q+q^2\,,q^2\,,q^2+1\}.
 \end{split}
\end{equation*}
\end{example}
\vskip0.5cm
\begin{figure}\vskip0.5cm
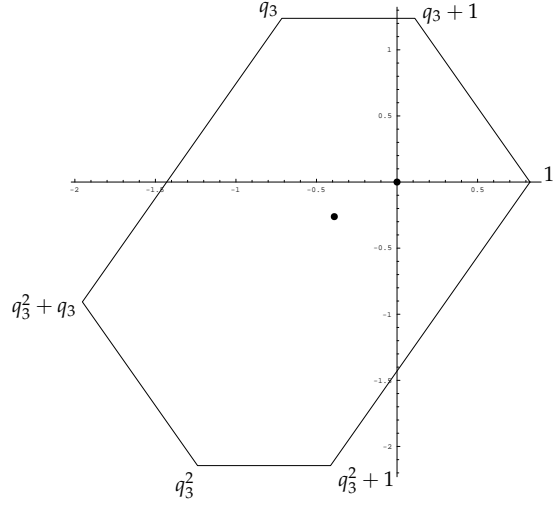
\caption{Extremal points of $P_{3,p}$ with $p=2^{1/3}$. When the alphabet is $\{0,1\}$, this coincides with the convex hull of $\Lambda_{A,n,p}$. }
\end{figure}

\begin{example}\label{t charcon ex 2}\black
 Set $p>1$ and $q:=q_{4,p}=pe^{\frac{2\pi}{4}i}$. By a direct computation, the set of extremal points f $P_{4,p}$, say $\mathcal E(P_{4,p})$, is
\begin{equation*}
 \begin{split}
 \mathcal E(P_{4,p})=\{1+q\,,q+q^2\,,q^2+q^3\,q^3+1\}.
 \end{split}
\end{equation*}
\end{example}\black

Example \ref{t charcon ex 1} and Example \ref{t charcon ex 2} suggest that the set of extremal points of $P_{n,p}$ has an internal structure. In general
  the vertices of $P_{n,p}$ are element of $X_{n,p}$ and in particular they are (finite) expansions in base $q_{n,p}$.
In order to make  the structure of the extremal points more evident, in next examples we focus on the sequences of binary coefficients associated to the extremal points.
 This point of view requires some notations.

\green\begin{notation}\black
 Let $(x_0\cdots x_{n-1})$ be a sequence in $\{0,1\}^n$ and  $q\in
 \CC$. We define
$$(x_0\cdots x_{n-1})_q:=\sum_{k=0}^{n-1}x_kq^k$$
and we introduce $\sigma$, the circular shift  on finite sequences:
$$\sigma (x_0x_1\cdots x_{n-1}):=(x_1\cdots x_{n-1}x_0).$$
The closure of $(x_0\cdots x_{n-1})$ with respect to $\sigma$ is denoted by
$$Orb(x_0\cdots x_{n-1}):=\{\sigma^h (x_0x_1\cdots x_{n-1})\mid h=0,\dots,n-1\}.$$
Finally we define
$$ Orb(x_0\cdots x_{n-1})_q:=\{\sigma^h (x_0x_1\cdots x_{n-1})_q \mid h=0,\dots,n-1\}.$$
\end{notation}\black

\green\begin{example}\label{t conchar ex 3}\black
Following relations hold for every $p>1$ and they are established by means of a symbolic computer program.

\noindent If $q=q_{3,p}$ then
\begin{align*}
\mathcal E(P_{3,p})&=\{1\,,1+q\,,q\,,q+q^2\,,q^2\,,q^2+1\}.\\
           &=\{(100)_{q},(110)_{q},(010)_{q},(011)_{q},(001)_{q},(101)_{q}\}\\
           &=Orb(100)_{q}\cup Orb(110)_{q};\\
\end{align*}
\noindent If $q=q_{4,p}$ then
\begin{align*}
\mathcal E(P_{4.p})&=\{1+q\,,q+q^2\,,q^2+q^3\,q^3+1\}.\\
           &=\{(1100)_{q},(0110)_{q},(0011)_{q},(1001)_{q}\}\\
           &=Orb(1100)_{q};
\end{align*}
\noindent If $q=q_{5,p}$ then
\begin{align*}
\mathcal E(P_{5,p})&=\{(11000)_{q},(11100)_{q},(01100)_{q},(01110)_{q},(00110)_{q},\\
          &\phantom{=\{}~(00111)_{q},(00011)_{q},(10011)_{q},(10001)_{q},(11001)_{q}\}\\
           &=Orb(1100)_{q}\cup Orb(11100)_{q};\\
\end{align*}
\noindent If $q=q_{6,p}$ then
\begin{equation*}
\begin{split}
\mathcal E(P_{6,p})&=\{(111000)_{q},(011100)_{q},(001110)_{q},\\
          &\phantom{=\{}~(000111)_{q},(100011)_{q},(110001)_{q}\}\\
           &=Orb(111000)_{q}.
\end{split}
\end{equation*}
\end{example}\black

In Example \ref{t conchar ex 3} the set of extremal points $\mathcal
E(P_{n,p})$ is shown to be intimately related to the sequences
$(1^{\lfloor n/2 \rfloor}0^{n-\lfloor n/2 \rfloor})$ and $(1^{\lceil
n/2 \rceil}0^{n-\lceil n/2 \rceil})$ when $n=3,5$ and to the
sequence $(1^{ n/2 }0^{n/2})$ when $n=4,6$. We now prove that this
is a general result. We set $q=q_{n,p}$ and for $h=1,\dots,n$ we
introduce the vertices
\begin{align}
\mathbf v_{2h-1}:=\sigma^h(1^{\lfloor n/2 \rfloor}0^{n-\lfloor n/2 \rfloor})&=\sum_{k=0}^{\lfloor n/2\rfloor-1}q^{k+h-1~\moda n}\\
                                                                          &=\sum_{k\in \underline K(h)}q^k
\end{align}
where $$\underline K(h):=\{k\in\{0,\dots,n-1\}\mid (k-h+1)~\moda n\leq \lfloor n/2\rfloor -1\};$$
\begin{align}
\mathbf v_{2h}:=\sigma^h(1^{\lceil n/2 \rceil}0^{n-\lceil n/2 \rceil})&=\sum_{k=0}^{\lceil n/2\rceil-1}q^{k-h+1~\moda n}\\
&=\sum_{k\in \overline K(h)}q^k
\end{align}
where
$$\overline K(h):=\{k\in\{0,\dots,n-1\}\mid (k-h+1)~ \moda n\leq \lceil n/2\rceil -1\}.$$

\green\begin{remark}\black
 By a direct computation, if $n$ is odd then for every $h=1,\dots,n$
\begin{align}
 \mathbf n_{2h-1}&=(-q^{h-2~\moda n})^\bot\label{n2hodd}\\
\mathbf n_{2h}&=(q^{h-2+\lceil n/2\rceil~\moda n})^\bot
\end{align}
while if $n$ is even then $\mathbf v_{2h}=\mathbf v_{2h-1}$ hence
\begin{align}
 \mathbf n_{2h-1}&=\left(1-\frac{1}{p^{n/2}}\right)(-q^{h-2~\moda n})^\bot\label{n2heven}\\
\mathbf n_{2h}&=0;
\end{align}
In view of (\ref{n2hodd}) and of (\ref{n2heven}), for every $n$ and every $h$
\begin{equation}\label{cos}
 \arg \mathbf n_{2h-1}=\left(\frac{(h-2)~\moda n}{n}2\pi+\frac{\pi}{2}\right)~ \moda 2\pi
\end{equation}
\end{remark}\black

\green\begin{lemma}\label{lnorm}\black
 For every $h=1,\dots,n$
\begin{equation}\label{Kh1}
 k\in \underline K(h) \text{ if and only if } q^k\cdot \mathbf n_{2h-1}\geq 0;
\end{equation}
\begin{equation}\label{Kh2}
 k\in \overline K(h) \text{ if and only if } q^k\cdot \mathbf n_{2h}\geq 0;
\end{equation}
\end{lemma}\black

\begin{proof}
For every $k=0,\dots,n-1$
\begin{equation}
  q^k\cdot \mathbf n_{2h-1}\geq 0
\end{equation}
if and only if
\begin{equation}\label{ineqcos1}
 \cos\left(\arg q^k-\arg \mathbf n_{2h-1}\right)=\cos\left(\frac{(k-h+1)~\moda n+1}{n}2\pi -\frac{\pi}{2}\right)\geq 0
\end{equation}
As $(k-h+1)~\moda n\in\{0,\dots,n-1\}$,
$$-\frac{\pi}{2}<\frac{(k-h+1)~\moda n+1}{n}2\pi -\frac{\pi}{2}\leq \frac{3\pi}{2},$$
 (\ref{ineqcos1}) reduces to
$$-\frac{\pi}{2}<\frac{(k-h+1)~\moda n+1}{n}2\pi -\frac{\pi}{2}\leq \frac{\pi}{2}$$
Hence (\ref{Kh1}) follows by the definition of $k\in \underline K(h)$ and by recalling that $(k-h+1)~\moda n\in\NN$.

To prove (\ref{Kh2}), we first remark that if $n$ is even then $\mathbf n_{2h}=0$ and (\ref{Kh2}) is immediate. If otherwise $n$ is odd then
$$\arg \mathbf n_{2h}=\left(\frac{(h-2+\lceil n/2\rceil)~\moda n}{n}2\pi + \frac{3\pi}{2}\right)~\moda 2\pi$$
therefore
\begin{equation}
  q^k\cdot \mathbf n_{2h-1}\geq 0
\end{equation}
if and only if
\begin{equation}\label{ineqcos2}
\cos\left(\frac{(k-h+1)~\moda n+1-\lceil n/2\rceil}{n}2\pi +\frac{\pi}{2}\right)\geq 0.
\end{equation}
As $(k-h+1)~\moda n\in\{0,\dots,n-1\}$, for every $k=0,\dots,n-1$
$$-\frac{\pi}{2}<\frac{(k-h+1)~\moda n+1}{n}2\pi -\frac{\pi}{2}\leq \frac{3\pi}{2},$$
 (\ref{ineqcos1}) reduces to
$$-\frac{\pi}{2}<\frac{(k-h+1)~\moda n+1}{n}2\pi -\frac{\pi}{2}\leq \frac{\pi}{2};$$
Relation (\ref{Kh2}) hence follows by the definition of $\overline K(h)$.
\end{proof}

\green\begin{lemma}\label{lco}\black
 For every $x\in X_{n,p}$ and for every $h=1,\dots,n$
\begin{equation}\label{n10}
 (x-\mathbf v_{2h-1})\cdot \mathbf n_{2h-1}\leq 0
\end{equation}
and
\begin{equation}\label{n20}
 (x-\mathbf v_{2h})\cdot \mathbf n_{2h}\leq 0.
\end{equation}
\end{lemma}\black
\begin{proof}
Let
$$x=\displaystyle{\sum_{k=0}^{n-1}x_kq^k\in X_{n,p}}$$
 with $x_0,\dots,x_{n-1}\in \{0,1\}$. Then for every $h=1,\dots, n$
\begin{align*}
 (x-\mathbf v_{2h-1})\cdot \mathbf n_{2h-1}&=\sum_{k\in \underline K(h)} (x_k-1)(q^k\cdot \mathbf n_{2h-1})+\sum_{k\not\in \underline K(h)}x_k(q^k\cdot \mathbf n_{2h-1}) \leq 0
\end{align*}
indeed Lemma \ref{lnorm} and $x_k\in\{0,1\}$ imply that all the terms of the above sums are non-positive. Similarly, again by Lemma \ref{lnorm} and by $x_k\in\{0,1\}$ we may deduce
\begin{align*}
 (x-\mathbf v_{2h-1})\cdot \mathbf n_{2h-1}&=\sum_{k\in \overline K(h)} (x_k-1)(q^k\cdot \mathbf n_{2h-1})+\sum_{k\not\in \overline K(h)}x_k(q^k\cdot \mathbf n_{2h-1}) \leq 0.
\end{align*}
\end{proof}

\yellow\begin{remark}\label{rmkP}\black
Proposition \ref{proconv} and Lemma \ref{lco} imply that $X_{n,p}$ is a subset of the polygon whose vertices are $\mathbf v_1,\dots \mathbf v_{2n}$, therefore
$$P_{n,p}=\conv (X_{n,p})\subseteq \conv\{\mathbf v_h\mid h=1,\dots,n\}.$$
This, together with the fact that $\mathbf v_h\in X_{n,p}$ for every $h=1,\dots,2n$, implies
\begin{equation}\label{pcon}
 P_{n,p}=\conv(\{\mathbf v_h\mid h=1,\dots,2n\}).
\end{equation}
In particular when $n$ is odd $\mathbf v_1,\dots,\mathbf v_{2n}$ coincide with the $2n$ vertices of $P_{n,p}$.
 When $n$ is even, $\mathbf v_1=\mathbf v_2,\dots,\mathbf v_{2n-1}=\mathbf v_{2n}$ are
the $n$ vertices of $P_{n,p}$.
\end{remark}

\begin{theorem}\black\label{t conchar}
 For every  $n\geq 1$, $p>1$ and for every finite alphabet $A$
\begin{equation}\label{key1}
 \text{\normalfont conv}(\Lambda_{n,p,A})=\frac{\max A-\min A}{p^n-1}\conv(\{\mathbf v_h\mid h=1,\dots,2n\})+\sum_{k=0}^{n-1}\min A~q^k.
\end{equation}
Moreover if $n$ is even then
\begin{equation}\label{key2}
 \text{\normalfont conv}(\Lambda_{n,p,A})=\frac{\max A-\min A}{p^n-1}\conv(\{\mathbf v_{2h}\mid h=1,\dots,n\})+\sum_{k=0}^{n-1}\min A~q^k.
\end{equation}
\end{theorem}\black

\begin{proof}
 In view of Remark \ref{rmkP},  (\ref{key1}) immediately follows by Proposition \ref{t convex p1}. while (\ref{key2})
 holds because if $n$ is even then $\mathbf v_{2h}=\mathbf v_{2h-1}$ for every $h=1,\dots,n$.
\end{proof}

\section{Representability in complex base}\label{s complex repr}
In this section we give a characterization of the convexity of $\Lambda_{n,p,A}$.
To this end we recall some basic elements of the Iterated Function System (IFS) theory.
An IFS $\mathcal F$ is a finite set of contractive maps over a metric space $\mathbb X$, in particular if $\mathbb X=\CC$ then
$$\mathcal F=\{f_i:\CC\to\CC|i=1,\dots, m\}$$
 for some $m\in \NN$  and for every $x,y\in\CC$ and $i=1,\dots,m$
  $$|f_i(x)-f_i(y)|\leq c_i |x-y|$$
for some $0<c_i<1$. The \emph{Hutchinson operator} acts on the power set of $\CC$ as follows
$$\mathcal F (X):=\bigcup_{i=1}^m f_i(X)=\bigcup_{i=1}^m\bigcup_{x\in X}f_i(x).$$
By the contraction principle, every IFS has a unique fixed point $R=\mathcal F(R)$ that can be constructed starting by any subset of $\CC$, indeed for every $X\subset \CC$
$$R=\lim_{k\to \infty} \mathcal F^k(X).$$
 \green \begin{example}\black
If $q\in \CC$ and if $|q|>1$ then
$$\mathcal F=\left\{f_i:x\mapsto \frac{1}{q}(x+a_i)\mid a_i\in \RR;~ i=1,\dots,m\right\}$$
is an IFS.
\end{example}\black
We define
\begin{equation}
\mathcal F_{q,A}:=\left\{f_i:x\mapsto \frac{1}{q}(x+a_i)\mid a_i\in A\right\}.
\end{equation}

\begin{figure}[t]
\centering
 \subfloat[]{\includegraphics[scale=0.2]{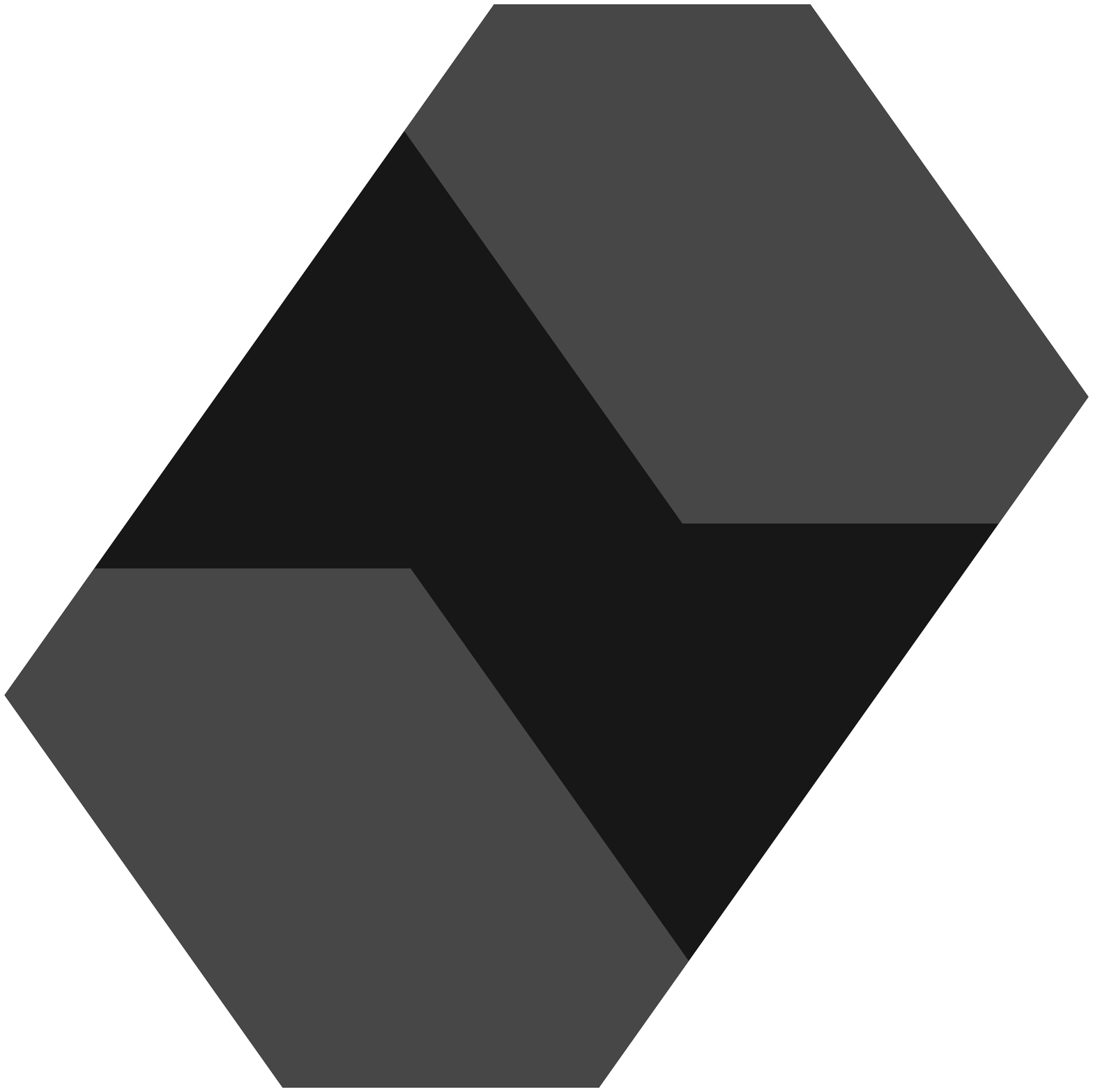}}
 \subfloat[]{\includegraphics[scale=0.2]{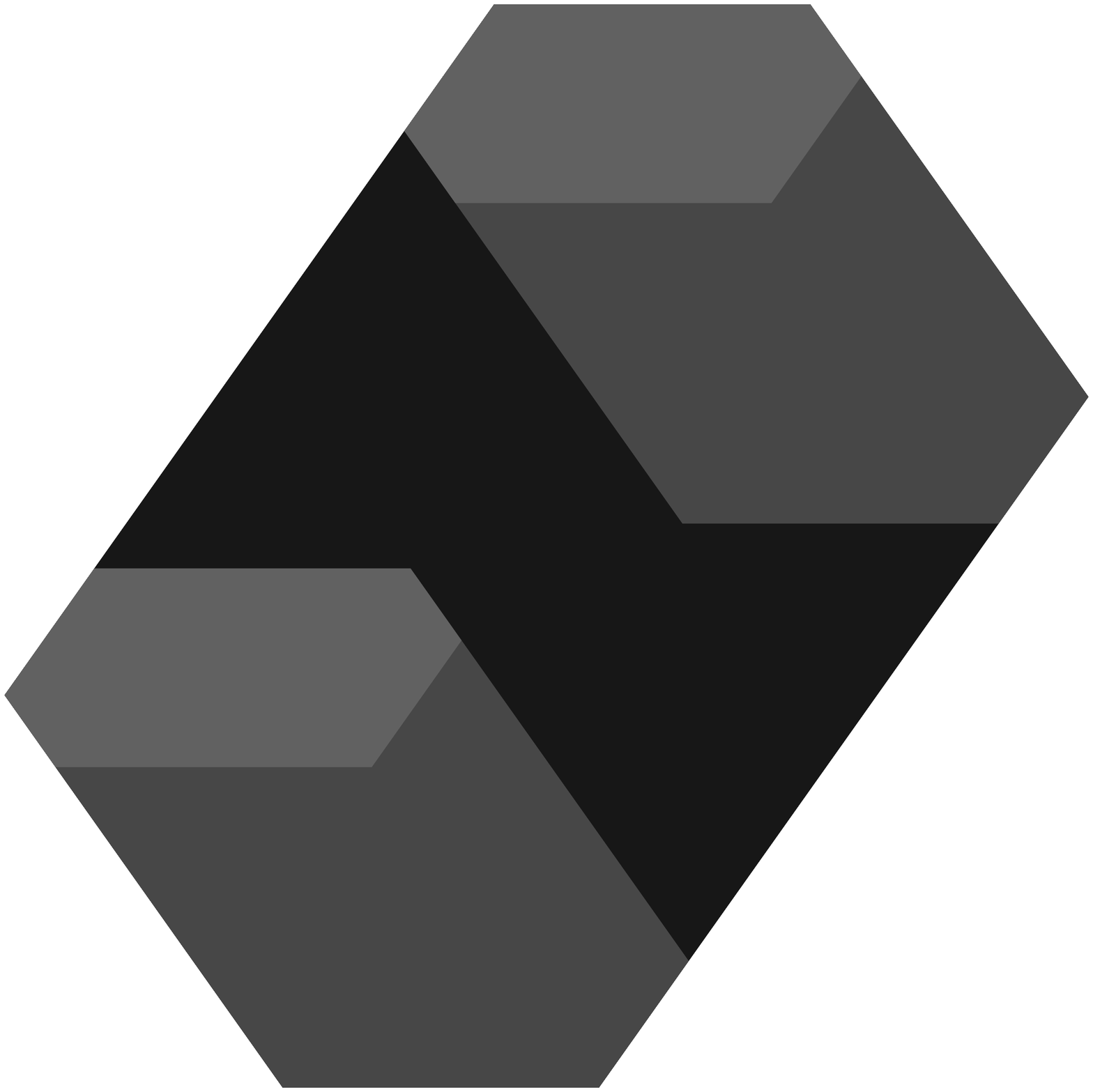}}
 \subfloat[]{\includegraphics[scale=0.2]{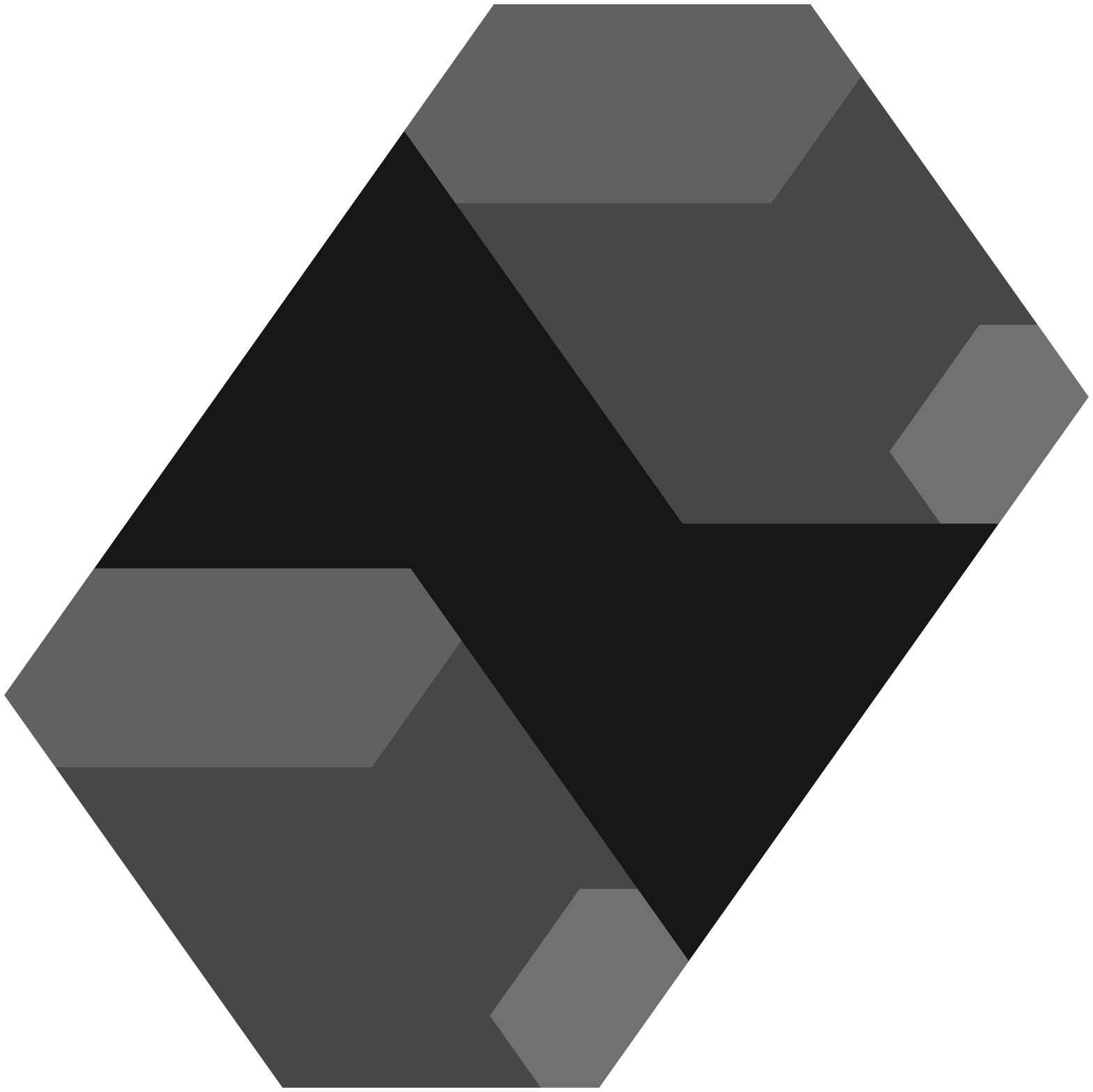}}\\
 \subfloat[]{\includegraphics[scale=0.2]{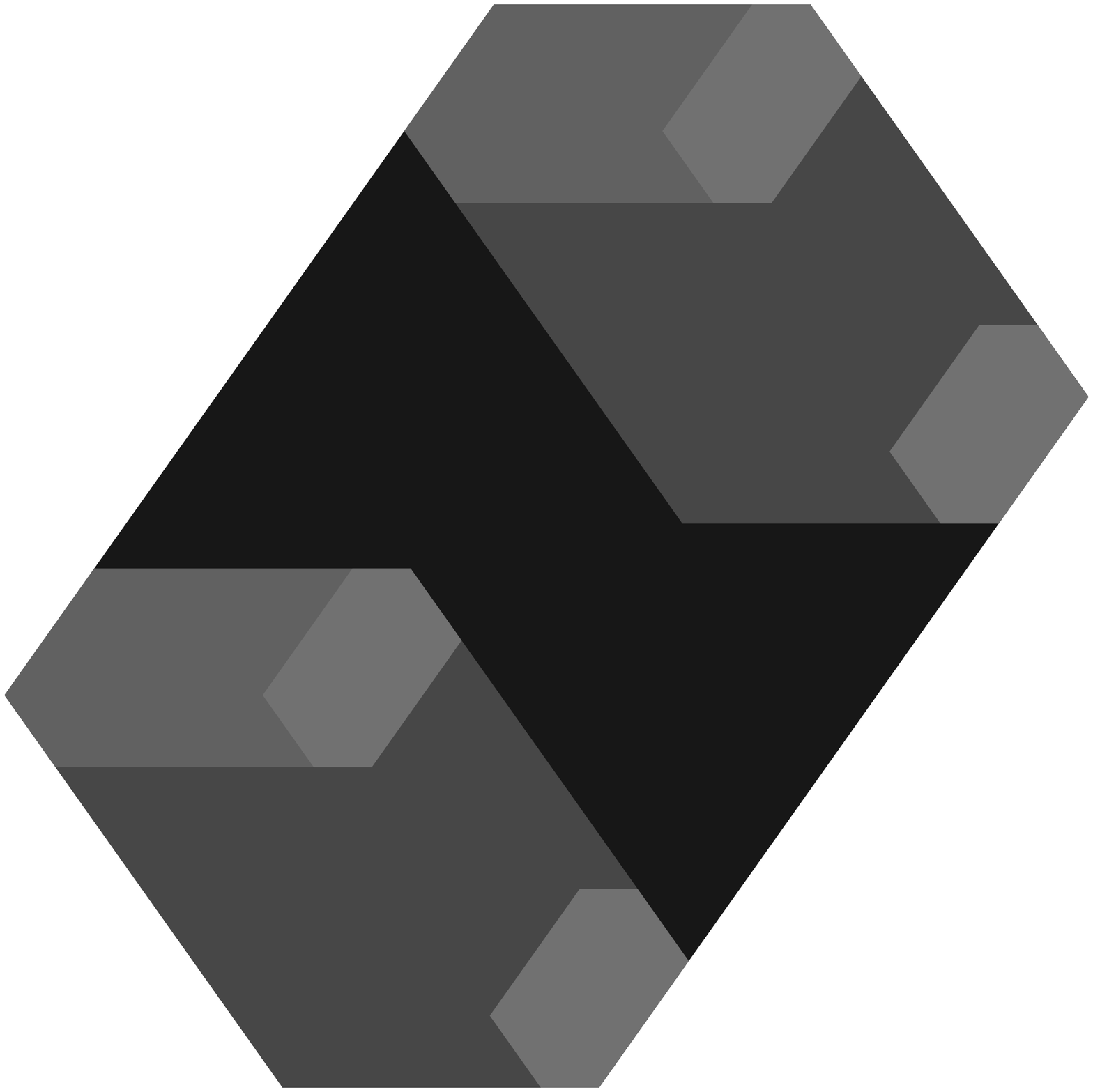}}
 \subfloat[]{\includegraphics[scale=0.2]{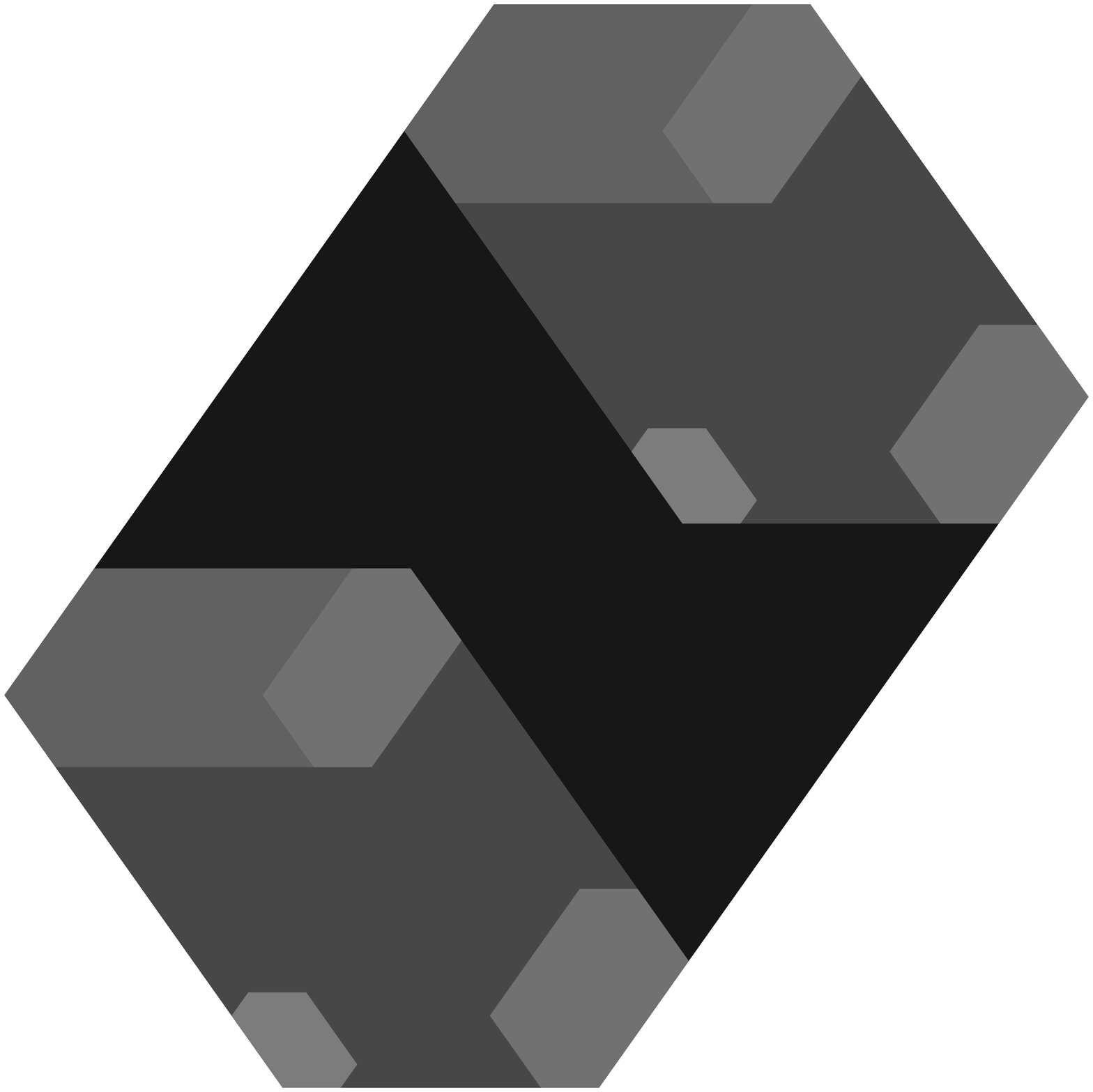}}
 \subfloat[]{\includegraphics[scale=0.2]{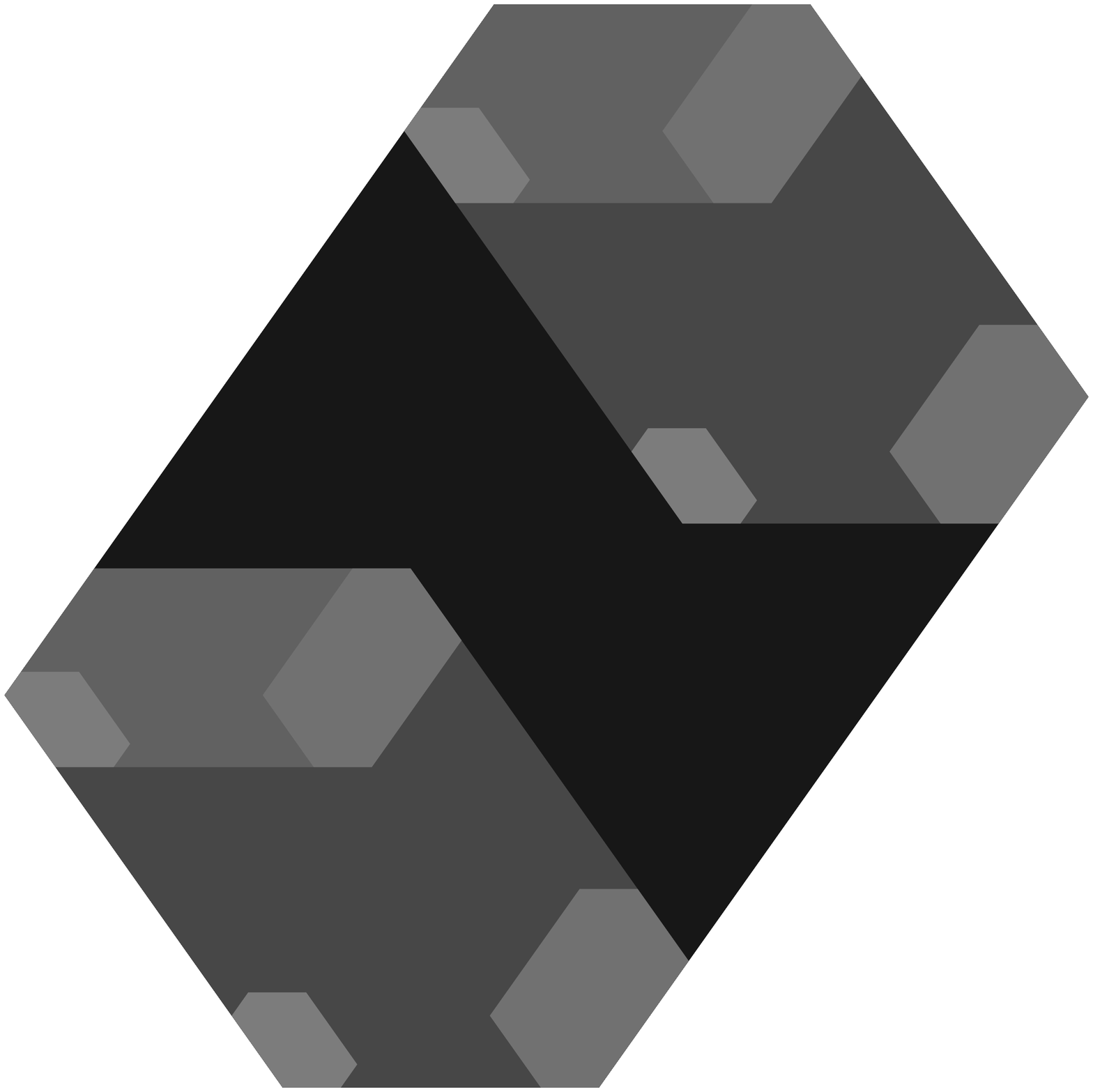}}
\caption{\label{f ifs3}First $6$ iterations of $F_{\{0,1\}}=\{x/q_{3,2}, 1/q_{3,2}(x+1)\}$ over $\text{\normalfont conv}(\Lambda_{3,2,\{0,1\}})$.}
\end{figure}

\green\begin{lemma}\black
Let $q\in \CC$, $|q|>1$ and $A=\{a_1,\dots,a_m\}\subset \RR$. Then the fixed point of $\mathcal F_{q,A}$
is
$$\displaystyle{\Lambda_{q,A}:=\left\{\sum_{k=1}^\infty \frac{x_k}{q^k}\mid x_k\in A\right\}}.$$
\end{lemma}\black
\begin{proof}
For every
$$x=\sum_{k=1}^\infty \frac{x_k}{q^k} \in \Lambda_{q,A}$$
 we have
\begin{equation*}
 f_i(x)=\frac{1}{q}x+\frac{1}{q}a_i=\sum_{i=2}^\infty \frac{x_{k-1}}{q^k}+\frac{a_i}{q}\in \Lambda_{q,A}.
\end{equation*}
Moreover if $x_1=a_i$ then
\begin{equation*}
 f^{-1}_i(x)=qx-a_i=x_1+q\sum_{k=2}^\infty \frac{x_{k}}{q^k}-a_i=\sum_{k=1}^\infty \frac{x_{k+1}}{q^i}\in \Lambda_{n,p,A}.
\end{equation*}
Therefore
$$\bigcup_{i=1}^mf_i(\Lambda_{q,A})=\Lambda_{q,A}$$
namely $\Lambda_{n,p,A}$ is the fixed point of $\mathcal F_{q,A}$.
\end{proof}
If $q=q_{n,p}$ we define
$$\mathcal F_{n,p,A}:=\mathcal F_{q,A}.$$

\yellow\begin{lemma}\label{lz1}\black
 For every $n\in\NN$, $p>1$ and $A\subset\RR$, with $|A|<\infty$, if $\mathcal F_{n,p,A}(\conv(\Lambda_{n,p,A}))$ is convex then
\begin{equation}
 \conv(\Lambda_{n,p,A})\subseteq \mathcal F_{n,p,A}(\conv(\Lambda_{n,p,A}).
\end{equation}

\end{lemma}\black
\begin{proof}
 Consider $\mathbf w_1,\dots,\mathbf w_{2n}$,  the (possibly pairwise coincident) vertices of $\Lambda_{n,p,A}$. By Theorem \ref{t conchar}, for every $h=1,\dots,2n$

\begin{align*}
 \mathbf w_h&=\frac{\max A-\min A}{p^n-1}\mathbf v_h+\frac{1}{p^n-1}\sum_{k=0}^{n-1}\min A ~q^k\\
            &=\frac{\max A-\min A}{p^n-1}\sigma^h(1^l0^{n-l})_q+\frac{1}{p^n-1}\sum_{k=0}^{n-1}\min A ~q^k\\
\end{align*}
for some $h\in\{1,\dots,n\}$ and $l\in\{\lfloor n/2\rfloor,\lceil n/2\lceil\}$. Remark that if $\varepsilon\in\{0,1\}$ is the last digit of $\sigma^h(1^l0^{n-1})$ then
$$q\sigma^h(1^l0^{n-l})_q=\sigma^{h+1}(1^l0^{n-l})_q+\varepsilon(1-q^n).$$
Now, as $q^n=p^n$ if $\varepsilon=0$ then
\begin{align*}
f^{-1}_1(\mathbf w_h)&= q\mathbf w_h-\min A\\
&=\frac{\max A-\min A}{p^n-1}(q\sigma^{h}(1^l0^{n-l})_q)+\frac{q}{p^n-1}\sum_{k=0}^{n-1}\min A ~ q^{k}-\min A\\
&=\frac{\max A-\min A}{p^n-1}(\sigma^{h+1}(1^l0^{n-l})_q)+\frac{1}{p^n-1}\sum_{k=0}^{n-1}\min A ~ q^{k}\\
&=\mathbf w_{h+2}.
\end{align*}
therefore $f_1(\mathbf w_{h+2})=\mathbf w_h$. By a similar argument, it is possible to show that if $\varepsilon=1$ then $f_m(\mathbf w_{h+2})=\mathbf w_h$. Thus
\begin{equation}\label{wh}
 \mathbf w_h\in \mathcal F_{q,A}(\{\mathbf w_h\mid h=1,\dots,2n\}
\end{equation}
and, by the arbitrariness of $\mathbf w_h$,
\begin{equation}\label{wh}
\{\mathbf w_h\mid h=1,\dots,2n\}\subseteq \mathcal F_{q,A}(\{\mathbf w_h\mid h=1,\dots,2n\}
\end{equation}

Hence
\begin{align*}
 \conv(\Lambda_{n,p,A})=&\conv(\{\mathbf w_h\mid h=1,\dots,2n\})\\
\subseteq&\conv\left(\mathcal F_{q,A}(\{\mathbf w_h\mid h=1,\dots,2n\})\right)\\
=&\conv\left(\bigcup_{i=1}^mf_i(\{\mathbf w_h\mid h=1,\dots,2n\})\right)\\
\subseteq&\conv\left(\bigcup_{i=1}^mf_i(\conv(\{\mathbf w_h\mid h=1,\dots,2n\}))\right)\\
=&\conv(\mathcal F_{n,p,A}(\conv(\Lambda_{n,p,A})))\\
=&\mathcal F_{n,p,A}(\conv(\Lambda_{n,p,A})).
\end{align*}

\end{proof}

\green\begin{lemma}\label{lz2}\black
 For every $n\in\NN$, $p>1$ and $A\subset\RR$
\begin{equation}
  \mathcal F_{n,p,A}(\conv(\Lambda_{n,p,A}))\subseteq \conv(\Lambda_{n,p,A}).
\end{equation}
 \end{lemma}\black
\begin{proof}
 Let $y\in  \mathcal F_{n,p,A}(\conv(\Lambda_{n,p,A}))$ so that $y=f_i(x)$ for some $i=1,\dots,m$ and some $x\in \conv(\Lambda_{n,p,A})$. In particular
$x=\alpha x_1+(1-\alpha)x_2$ for some $x_1,x_2\in \Lambda_{n,p,A}$ and, consequently,
$$y=\alpha f_i(x_1)+(1-\alpha)f_i(x_2)$$
because $f_i$ is a linear map. Since $\Lambda_{n,p,q}$ is the fixed point of an IFS containing $f_i$ we have $f_i(x_1),f_i(x_2)\in \Lambda_{n,p,A}$ and, consequently,
$y\in \conv(\Lambda_{n,p,A})$. Thesis follows by the arbitrariness of $y$.
\end{proof}
\begin{lemma}\label{lz3}
 For every $n\in\NN$, $p>1$ and $A\subset\RR$, $\Lambda_{n,p,A}$ is convex if and only if $\mathcal F_{n,p,A}(\conv (\Lambda_{n,p,A}))$ is convex.
\end{lemma}

\begin{proof}
 If $\Lambda_{n,p,A}$ is convex, then the convexity of $\mathcal F_{n,p,A}(\conv(\Lambda_{n,p,A}))$ follows by
\begin{equation}
\conv(\Lambda_{n,p,A})=\Lambda_{n,p,A}=\mathcal F_{n,p,A}(\Lambda_{n,p,A})=\mathcal F_{n,p,A}(\conv(\Lambda_{n,p,A})).
\end{equation}

If otherwise $\mathcal F_{n,p,A}(\conv (\Lambda_{n,p,A}))$ is convex, then by Lemma \ref{lz1}
$$\conv (\Lambda_{n,p,A})\subseteq \mathcal F_{n,p,A}(\conv (\Lambda_{n,p,A}))$$
while Lemma \ref{lz2} implies
$$\mathcal F_{n,p,A}(\conv (\Lambda_{n,p,A}))\subseteq \conv (\Lambda_{n,p,A});$$
therefore
$$\mathcal F_{n,p,A}(\conv (\Lambda_{n,p,A}))=\conv (\Lambda_{n,p,A});$$
and by the uniqueness of the fixed point of $\mathcal F_{n,p,A}$
$$\conv (\Lambda_{n,p,A})=\Lambda_{n,p,A}.$$
\end{proof}

\begin{figure}[h]
\centering
 \subfloat[$\Lambda_{3,2^{1/3},\{0,1\}}$]{\includegraphics[scale=0.5]{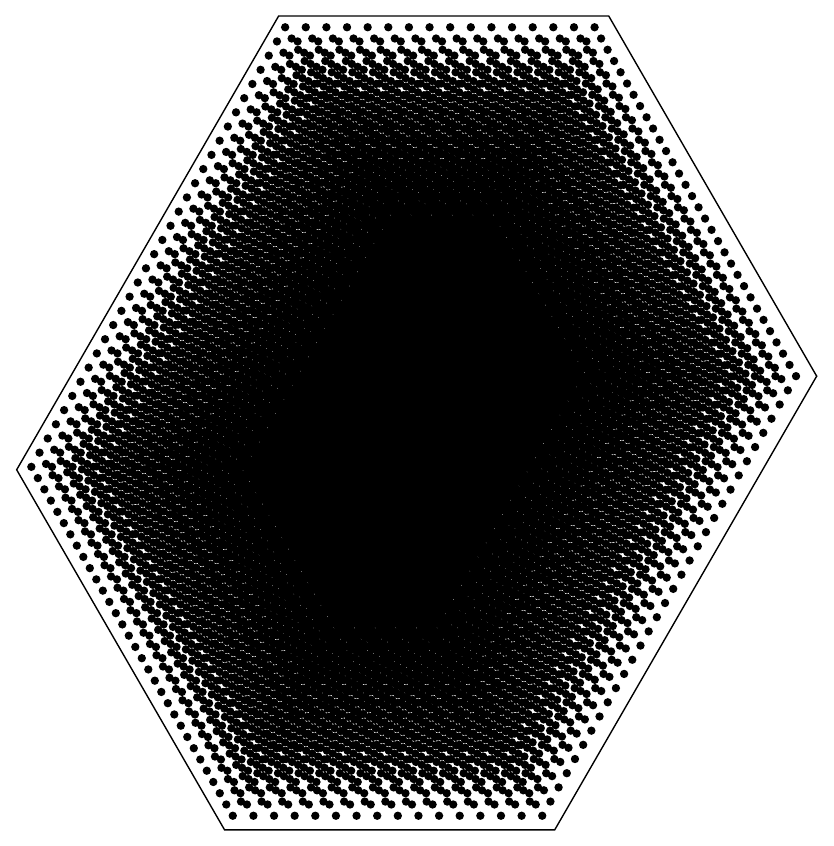}}
 \subfloat[$\Lambda_{3,2^{1/3}+0.1,\{0,1\}}$]{\includegraphics[scale=0.5]{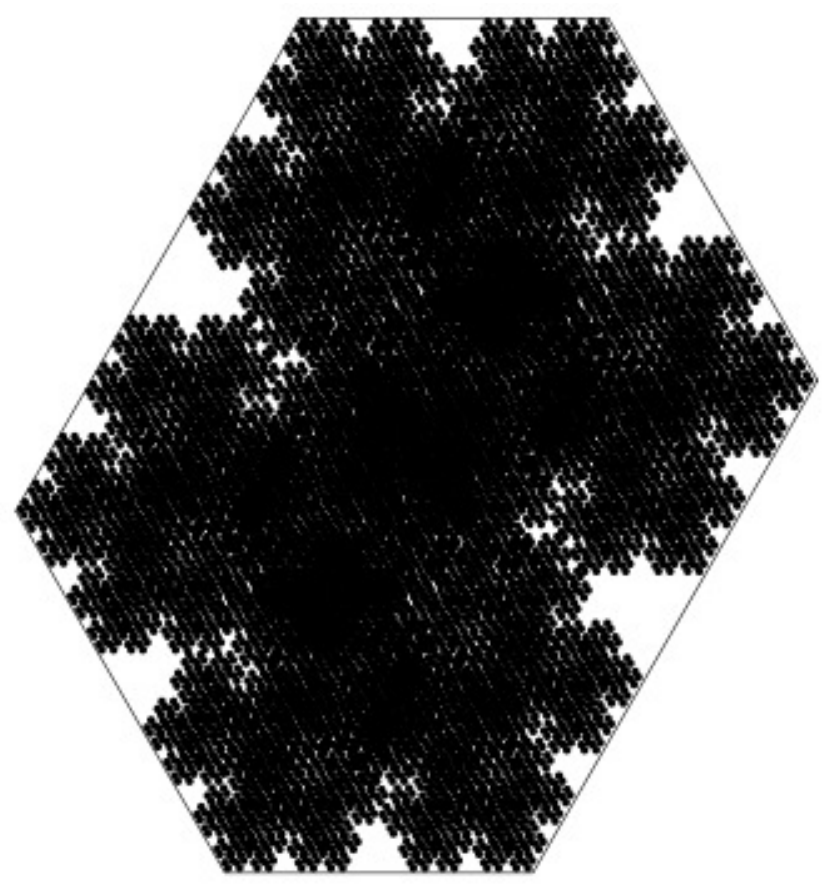}}
 \subfloat[$\Lambda_{3,2^{1/3}+0.2,\{0,1\}}$]{\includegraphics[scale=0.5]{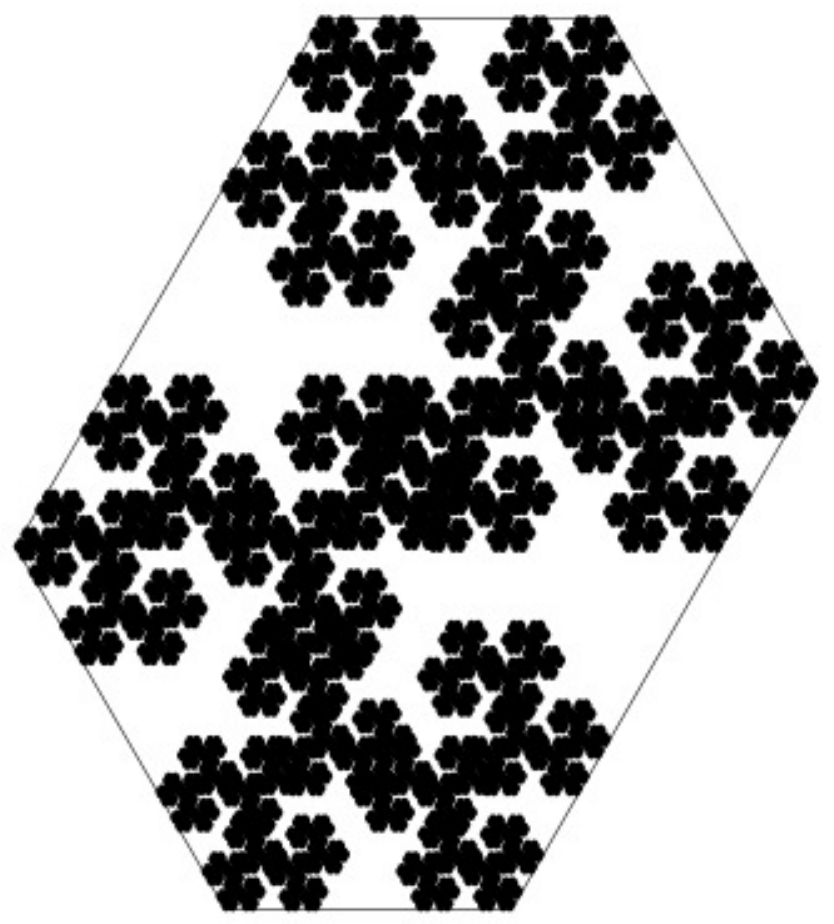}}
\\
 \subfloat[$\Lambda_{3,2^{1/3}+0.3,\{0,1\}}$]{\includegraphics[scale=0.5]{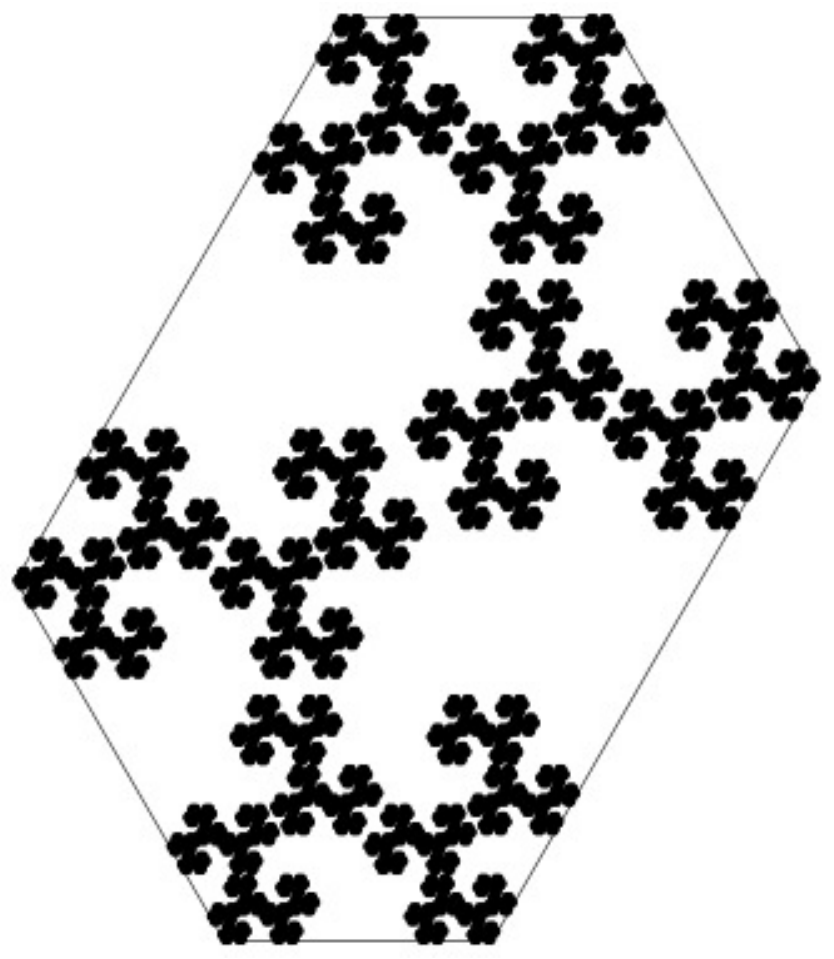}}
 \subfloat[$\Lambda_{3,2^{1/3}+0.4,\{0,1\}}$]{\includegraphics[scale=0.5]{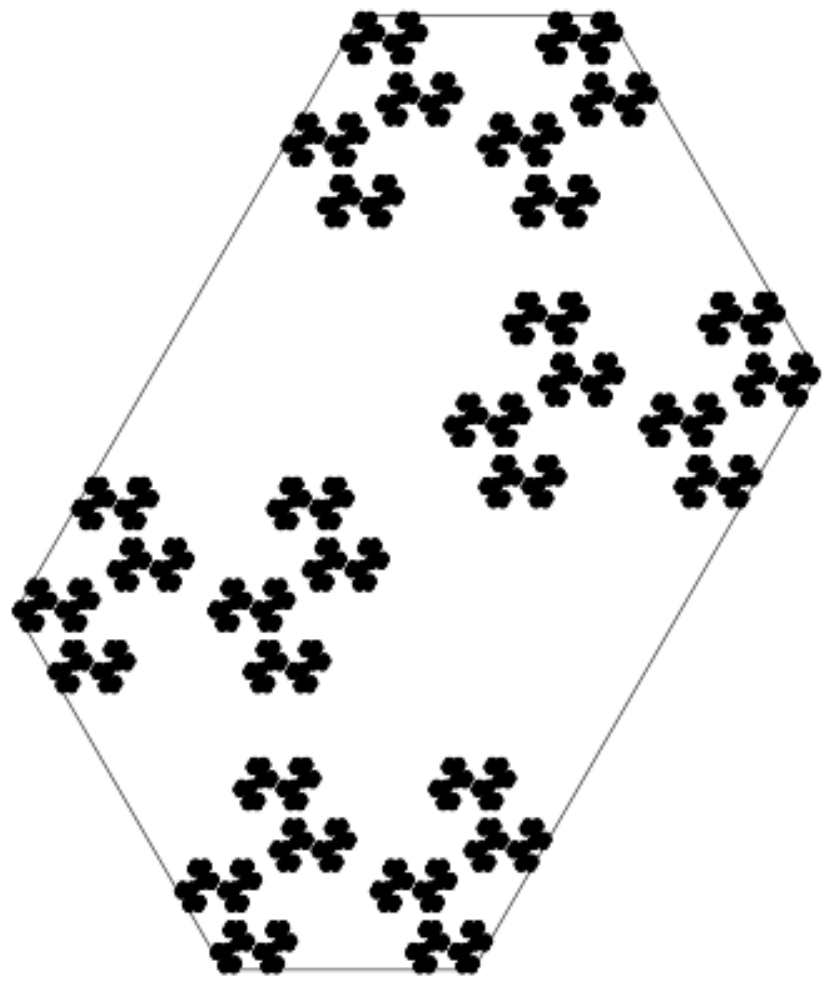}}
 \subfloat[$\Lambda_{3,2^{1/3}+0.5,\{0,1\}}$]{\includegraphics[scale=0.5]{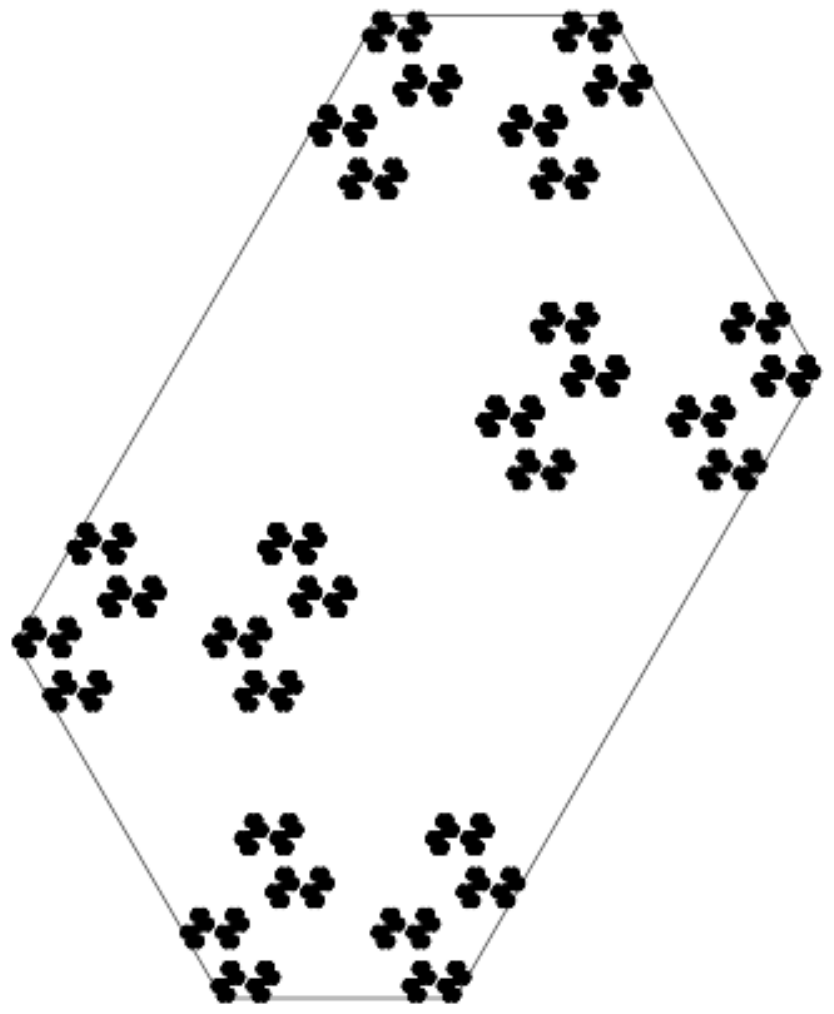}}

\caption{$\Lambda_{3,2^{1/3}+0.1 k,\{0,1\}}$, with $k=0,\dots,5$, is approximated with the set of expansions with length $14$.}
\end{figure}
\yellow
\begin{theorem}\label{t reprcon}\black
 The set of representable numbers in base $q_{n,p}$ and alphabet $A=\{a_i\mid i=1,\dots,m\}$ is convex if and only if
\begin{equation}\label{t reprcon e0}
 \max_{i=1,\dots,m-1}{a_{i+1}-a_i}\leq \frac{\max A-\min A}{p^n-1}.
\end{equation}
\end{theorem}\black

\begin{proof}
By applying Proposition \ref{l:geometrical hq} with $P=P_{n,p}$ and
$$t_i=a_i\frac{p^n-1}{\max A-\min A},$$
we get that (\ref{t reprcon e0}) holds if and only if
$$\bigcup_{i=1}^m (P_{n,p}+t_i)$$
is convex. Then (\ref{t reprcon e0}) is equivalent to the convexity of
\begin{align*}\label{52}
\overline P:= &\frac{\max A-\min A}{q(p^n-1)}\bigcup_{i=1}^m (P_{n,p}+t_i)+\frac{1}{q(p^n-1)}\sum_{k=0}^{n-1}\min A~q^k.\\
              =&\bigcup_{i=1}^m\frac{1}{q}\left(\frac{\max A-\min A}{q(p^n-1)}P_{n,p}+\frac{1}{q(p^n-1)}\sum_{k=0}^{n-1}\min A~q^k+a_i\right)
\end{align*}
By Theorem \ref{t conchar} we have
\begin{align*}
\overline P=\bigcup_{i=1}^m\frac{1}{q}\left(\conv(\Lambda_{n,p,A})+a_i\right)=\mathcal F_{n,p,A}(\conv(\Lambda_{n,p,A})),
\end{align*}
therefore $\mathcal F_{n,p,A}(\conv(\Lambda_{n,p,A}))$ is convex if and only if (\ref{t reprcon e0}) holds. Thesis hence follows by Lemma \ref{lz3}.
%
%
\end{proof}

\green\begin{corollary}\black
 Let $A=\{a_i\mid i=1,\dots,m\}$. If
$$ \max_{i=1,\dots,m-1}{a_{i+1}-a_i}\leq \frac{\max A-\min A}{p^n-1}$$ then every
$$x\in \frac{\max A-\min A}{p^n-1}P_{n,p}+\sum_{k=0}^{n-1}\min A q^k_{n,p}$$
 has a representation in base $q_{n,p}$ and alphabet $A$.
\end{corollary}\black
\begin{proof}
 It immediately follows by Theorem \ref{t reprcon}.
\end{proof}

\begin{example}\black
 If $p=2^{1/n}$ and $A=\{0,1\}$ then $\Lambda_{n,p,A}$ is a convex set coinciding with $P_{n,p}$. In particular $\Lambda_{n,p,A}$ is an $2n$-gon if $n$ is odd and it is an $n$-gon if $n$ is even.
\end{example}

\begin{example}\black
 If $A=\{0,1,\dots,\lfloor p^n\rfloor \}$ then $\Lambda_{n,p,A}$ is a convex set or, equivalently,
$$\frac{\lfloor p^n\rfloor}{p^n-1} P_{n,p}$$
 is completely representable.
\end{example}\black


%

\begin{thebibliography}{10}
\bibitem{AT04}
S.~Akiyama and J.~M. Thuswaldner.
 A survey on topological properties of tiles related to number
  systems.
 {\em Geom. Dedicata}, 109:89--105, 2004.

\bibitem{DK88}
Z.~Dar{\'o}czy and I.~K{\'a}tai.
\newblock Generalized number systems in the complex plane.
\newblock {\em Acta Math. Hungar.}, 51(3-4):409--416, 1988.
\bibitem{DJM85}
V.~A. Dimitrov, G.~A. Jullien, and Miller W.C.
\newblock Complexity and fast algorithms for multiexponentiation.
\newblock 31:141--147, 1985.
\bibitem{FS03}
Ch. Frougny and A.~Surarerks.
\newblock On-line multiplication in real and complex base.
\newblock {\em Computer Arithmetic, IEEE Symposium on}, 0:212, 2003.
\bibitem{Gil81}
W.~J. Gilbert.
\newblock Geometry of radix representations.
\newblock In {\em The geometric vein}, pages 129--139. Springer, New York,
  1981.
\bibitem{Gil84}
W.~J. Gilbert.
\newblock Arithmetic in complex bases.
\newblock {\em Math. Mag.}, 57(2):77--81, 1984.
\bibitem{Gil87}
W.~J. Gilbert.
\newblock Complex bases and fractal similarity.
\newblock {\em Ann. Sci. Math. Qu\'ebec}, 11(1):65--77, 1987.
\bibitem{Gems}
editor Heckbert, P.~S.
\newblock {\em Graphics Gems IV}.
\newblock Academic Press, 1994.
\bibitem{KK81}
I.~K{\'a}tai and B.~Kov{\'a}cs.
\newblock Canonical number systems in imaginary quadratic fields.
\newblock {\em Acta Math. Acad. Sci. Hungar.}, 37(1-3):159--164, 1981.
\bibitem{KS75}
I.~K{\'a}tai and J.~Szab{\'o}.
\newblock Canonical number systems for complex integers.
\newblock {\em Acta Sci. Math. (Szeged)}, 37(3-4):255--260, 1975.
\bibitem{Knu60}
D.~E. Knuth.
\newblock An imaginary number system.
\newblock {\em Comm. ACM}, 3:245--247, 1960.
\bibitem{Knu71}
D.~E. Knuth.
\newblock {\em The art of computer programming. {V}ol. 2}.
\newblock Addison-Wesley Publishing Co., Reading, Mass., first edition, 1971.
\newblock Seminumerical algorithms, Addison-Wesley Series in Computer Science
  and Information Processing.
\bibitem{KL07}
V.~Komornik and P.~Loreti.
\newblock Expansions in complex bases.
\newblock {\em Canad. Math. Bull.}, 50(3):399--408, 2007.
\bibitem{Pen65}
W.~Penney.
\newblock A ``binary'' system for complex numbers.
\newblock {\em J. ACM}, 12(2):247--248, 1965.
\bibitem{Pic02}
D.~Pich\'e.
\newblock Complex bases, number systems and their application to the fractal
  wavelet image coding.
\newblock {\em PhD thesis (Univ. Waterloo, Ontario}, 0, 2002.
\bibitem{PIP}
J.~Pineda.
\newblock A parallel algorithm for polygon rasterization.
\newblock 22(4):17--20, 1988.

\bibitem{Sol00}
J.~A. Solinas.
\newblock Efficient arithmetic on {K}oblitz curves.
\newblock {\em Des. Codes Cryptogr.}, 19(2-3):195--249, 2000.
\newblock Towards a quarter-century of public key cryptography.
\end{thebibliography}
\end{document}